\documentclass{article}

\usepackage[main, final]{neurips_2026}

\usepackage[utf8]{inputenc}
\usepackage[T1]{fontenc}
\usepackage{microtype}
\usepackage{graphicx}
\usepackage{subcaption}
\usepackage{booktabs}
\usepackage{float}
\usepackage{algorithm}
\usepackage{xcolor}
\usepackage{hyperref}
\usepackage{url}

\usepackage{amsmath}
\usepackage{amssymb}
\usepackage{amsfonts}
\usepackage{mathtools}
\usepackage{amsthm}

\usepackage{multirow}
\usepackage{kotex}

\usepackage[capitalize,noabbrev]{cleveref}

\theoremstyle{plain}
\newtheorem{theorem}{Theorem}[section]
\newtheorem{proposition}[theorem]{Proposition}
\newtheorem{lemma}[theorem]{Lemma}
\newtheorem{corollary}[theorem]{Corollary}
\theoremstyle{definition}
\newtheorem{definition}[theorem]{Definition}
\newtheorem{assumption}[theorem]{Assumption}
\theoremstyle{remark}

\newtheorem{claim}[theorem]{Claim}

\providecommand{\todo}[2][]{}

\author{
  Donghyun Park\\
  Department of Statistics\\
  Seoul National University\\
  \texttt{joshua8699@snu.ac.kr}
  \And
  Junhyun An\\
  Department of Statistics\\
  Seoul National University\\
  \texttt{anjunhyun@snu.ac.kr}
  \And
  Taehyoung Kim\\
  Department of Statistics\\
  Seoul National University\\
  \texttt{lion384@snu.ac.kr}
  \And
  Jisu Kim\\
  Department of Statistics\\
  Seoul National University\\
  \texttt{jkim82133@snu.ac.kr}
}

\allowdisplaybreaks

\title{Topological Periodicity Test (TopPT) via\\
Confidence Bound of Time-Delay Embeddings}

\begin{document}

\maketitle

\begin{abstract}
Time-delay embedding is a fundamental technique in Topological Data Analysis (TDA) for reconstructing phase-space dynamics of time-series data, where persistent homology can reveal loops associated with periodicity. However, rigorous statistical uncertainty quantification for these features remains underdeveloped.
First, we analyze the topology of time-delay embeddings, showing that the embedded trajectory is homotopy equivalent to a circle ($S^1$) for periodic signals and contractible for non-periodic ones. We also prove a positive lower bound on the embedding reach, ensuring stable topological features.
Second, we develop a subsampling approach to construct confidence bounds for persistence diagrams. Under standard manifold regularity conditions, we derive data-dependent bounds with asymptotic guarantees.
Finally, we propose Topological Periodicity Test (TopPT), a hypothesis testing framework for periodicity with asymptotically controlled type I and type II error rates.
Experiments on bounded-error synthetic data show that raw TDA detects periodic alternatives while avoiding false rejections on structured non-periodic signals, and that the robust rule is conservative after interpolation-error correction. On PhysioNet Fantasia and BIDMC respiratory waveforms, TDA detects most records but only a subset of local windows, unlike scalar periodogram baselines.

\end{abstract}

\section{Introduction}
Time series data consist of observations ordered in time, and understanding their underlying dynamics is a central problem in many scientific domains. In particular, detecting periodicity is crucial in applications where signals exhibit approximately repeating patterns. Classical approaches to periodicity detection primarily rely on spectral methods, such as the Generalized Lomb--Scargle periodogram~\cite{Zechmeister2009}, and autocorrelation-based techniques.



While these methods are effective for identifying frequency components, they may fail to capture geometric structures of underlying data. In contrast, Topological Data Analysis (TDA) provides a complementary perspective by focusing on the geometry of the data. Persistent homology summarizes topological features such as connected components and loops through persistence diagrams~\cite{PunXiaLee2018}, and has proven effective in time-series analysis, where one-dimensional features capture cyclic behavior~\cite{PEREIRA20156026}. The SW1Pers method~\cite{Perea2015} formalizes this idea using sliding window embeddings, in which loops in the embedded trajectory correspond to periodic or quasi-periodic dynamics. The method also establishes convergence guarantees for persistence diagrams as the embedding dimension increases.




Despite these advances, two fundamental challenges remain. First, there is no rigorous theoretical characterization of how periodicity determines the topology of sliding window embeddings. Second, existing approaches lack a statistical framework for quantifying uncertainty in persistence diagrams, limiting their use for formal inference.



\textbf{Our contributions}

First, we provide a theoretical characterization of time-delay embeddings, establishing a fundamental distinction between periodic and non-periodic signals. We prove that the embedded trajectory is homotopy equivalent to a circle ($S^1$) for periodic signals and to an interval for non-periodic ones. Moreover, we show that the reach of the sliding window embedding admits a positive lower bound, ensuring that persistent homology captures a stable topological feature with a nontrivial lifespan.

Second, we develop a statistical inference framework for persistence diagrams based on subsampling. By analyzing the Hausdorff distance between subsamples and the underlying support, we derive a data-dependent confidence radius with asymptotic guarantees. These results enable principled uncertainty quantification for topological features associated with periodicity.

Third, we propose a statistically valid hypothesis test for periodicity with controlled error rates. We show that the test achieves the prescribed significance level $\alpha$ and that both type~I and type~II errors are controlled under suitable conditions. Furthermore, we establish that these conditions are satisfied asymptotically, yielding consistency of the proposed test.

\section{Preliminaries}

In this section, we introduce the mathematical and statistical framework for our analysis. First, we describe time-delay embedding, which constructs point clouds from time series data. Then, we review persistent homology as a tool for extracting topological features.


\subsection{Time-delay Embedding}
Time-delay embedding is a method for reconstructing the phase space of a time-series, which is particularly effective for periodic data~\cite{PhysRevLett.45.712, Perea2015, DBLP:books/springer/Takens81}. Given a continuous function \(f\in C(\mathbb{T},\mathbb{R})\), for an integer \(m\) and real number \(\tau\) the sliding window embedding \(SW_{m,\tau}:C(\mathbb{T},\mathbb{R})\to C( \mathbb{T},\mathbb{R}^{m+1})\) of \(f\) is defined as:
$SW_{m,\tau}f(t) = 
[
f(t), 
f(t+\tau), 
\cdots 
,f(t+m\tau)
]^{\top}\in \mathbb{R}^{m+1}$
.
This process maps the time series to a point cloud \(X\) in high dimensional space \(\mathbb{R}^{m+1}\). The embedding relies on two main parameters: the embedding dimension \(m\) and the time delay \(\tau\). Choosing appropriate values for \(m\) and \(\tau\) is crucial to capture properties of time-series such as periodicity.

For a purely periodic signal, the reconstructed point cloud \(X\) generically traces a limit loop that is topologically equivalent to a one-dimensional circle \(S^1\)~\cite{DBLP:books/cup/Robinson11, DBLP:books/springer/Takens81,Perea2015}. In practice, noise and finite sampling can thicken or distort this manifold, but its essential topology remains a robust descriptor of the system's periodic nature.


\subsection{Topological Data Analysis (TDA)}

Topological Data Analysis (TDA) is a framework rooted in algebraic topology and computational geometry that provides structural properties of the data. It relies on the multi-scale analysis of topological features, a concept formalized through persistent homology~\cite{Edelsbrunner2002,zomorodian2005computing}. 

\subsubsection{Persistent Homology, filtrations, and diagrams}

In topological data analysis, algebraic topology is employed to characterize the intrinsic shape of data. For a given space, its topological features are described by homology groups, where the $k$-th Betti number $\beta_k$ counts the number of $k$-dimensional features. Specifically, $\beta_0$ represents the number of connected components, and $\beta_1$ denotes the number of 1-dimensional loops~\cite{Munkres1984}.

To infer the underlying topology of a point cloud, we construct a filtration parameterized by a scale $\epsilon \ge 0$. While the Čech filtration $\mathcal{C}(\epsilon)$ accurately captures the exact topology of the $\epsilon$-balls covering the data, it is computationally prohibitive. In practice, we utilize the Vietoris-Rips filtration $\mathcal{VR}(\epsilon)$, which efficiently constructs complexes using only pairwise distances. This substitution is theoretically justified by the tightly interleaved relationship between the two filtrations:
\[
\mathcal{C}(\epsilon) \subseteq \mathcal{VR}(2\epsilon) \subseteq \mathcal{C}(2\epsilon).
\]

As the scale parameter $\epsilon$ increases within a filtration, new topological features form (birth, denoted by $b$) and eventually merge or fill in (death, denoted by $d$). Persistent homology algebraically tracks these evolutionary lifespans~\cite{Edelsbrunner2002}.
The evolution of these homological features is concisely summarized in a persistence diagram $\mathcal{P}$. This diagram is a multiset of points $(b, d) \in \mathbb{R}^2$ with $b < d$, where each point represents a topological feature. The difference $d - b$ is referred to as the persistence of the feature. In practice, points located far from the diagonal $\Delta = \{(x, x) \mid x \in \mathbb{R}\}$ represent robust, prominent topological structures of the underlying space, whereas points close to the diagonal are typically regarded as sampling noise or minor irregularities.

A foundational property of persistence diagrams is its stability under perturbations~\cite{CohenSteiner2005}. For persistence diagrams constructed from Čech or Vietoris-Rips filtrations on discrete point clouds, the stability theorem implies that the bottleneck distance is strictly bounded by the Hausdorff distance $d_H$ between the underlying sets:
$
d_B(\mathcal{P}(X), \mathcal{P}(Y)) \le d_H(X, Y).
$

We refer to Appendix \ref{sec:simplicial_complex} for formal definitions of persistent homology, filtrations, and persistent diagrams.

\subsubsection{Confidence bounds for Persistence Diagrams}\label{sec:Confidence_bound_Preliminary}

There has been extensive research on formulating confidence bounds for persistence diagrams.~\cite{Fasy_2014} introduced a subsampling-based confidence bound construction. Assuming that we observe a sample $S_n = \{X_1, \cdots, X_n\}$ from a distribution $P$ concentrated on a set $\mathbb{M}$,~\cite{Fasy_2014} proved that under regularity assumptions, subsampling-based confidence bounds are valid. When the subsample size $b$ increases to infinity at a rate $b = o(n/\log n)$, they defined the function
$L(t) = \frac{1}{\binom{n}{b}} \sum_{j=1}^{\binom{n}{b}} I(d_H(S_{b,n}^{(j)},S_n)>t)$
where $S_{b,n}^{(j)}$ denotes the $j$-th subsample of size $b$ of $S_n$.

Defining $c_{\alpha} = 2L^{-1}(\alpha)$, let $\hat{\mathcal{P}}$ be the persistence diagram of $\{X_1, \cdots ,X_n\}$ and $\mathcal{P}$ be the persistence diagram of $\mathbb{M}$. The following theorem holds. See Section~\ref{sec:confidence_subsampling_formal} for formal assumptions and statements.
\begin{theorem}[\cite{Fasy_2014} Theorem 3]
\label{thm:confidenceBound}
    Under the regularity condition of $\mathbb{M}$, for all large $n$,
    \[\mathbb{P}(d_B(\hat{\mathcal{P}},\mathcal{P}) > c_{\alpha}) \le \alpha + 
        O\big((b/n)^{1/4}\big).
    \]
\end{theorem}

\section{Assumptions and Statistical Models}\label{sec:assumption}


We assume that there exists a true sampling function $f:\mathbb{R}\to\mathbb{R}$.
We sample points $t_{1},\ldots,t_{n}$ i.i.d. from a distribution $Q$
on $\mathbb{T}=[T_{\min},T_{\max}]$. For any $m$ and $\tau$, we assume that we can observe 
$\{SW_{m,\tau}f(t_{j})\}_{1\leq j\leq n}\subset \mathbb{R}^{m+1}$. Since the window uses values up to time $t+m\tau$, we write
$J_{m,\tau} := [T_{\min}, T_{\max} + m\tau]$
for the interval on which the signal must satisfy the analytic assumptions.

\begin{assumption}[Sampling density assumption]
\label{ass:Sampling_Density}
The sampling distribution $Q$ has a density $q$ with respect to Lebesgue measure on $\mathbb{T}$, and there exist constants $0<q_{\min}\le q_{\max}<\infty$ such that
$q_{\min}\le q(t)\le q_{\max}$ for all $t\in \mathbb{T}$. Moreover, $q$ is Lipschitz:
$\vert q(s)-q(t) \vert \le K_q \vert s-t\vert$ for all $s,t\in \mathbb{T}$.
\end{assumption}

To characterize the recurrent behavior of a function, we distinguish between periodic and non-periodic functions using the phase trajectory $\Phi_f(t) = (f(t),f'(t))$.

\begin{definition}
Let $f:\mathbb{R}\to\mathbb{R}$ be a differentiable function, and
fix $\Xi>0$. We say $f$ is $\Xi$-periodic if for any $t\in\mathbb{R}$, $f(t+\Xi)=f(t)$.
\medskip
We say $f$ is non-periodic on an interval $I$ if $\Phi_f$ is injective on $I$; for any $t_{1}\neq t_{2}\in I$, 
$f(t_{1})\neq f(t_{2})\text{ or } f'(t_{1})\neq f'(t_{2})$.
\end{definition}




However, even if we have valid confidence interval for the sampling function, sometimes non-periodic functions can behave very similarly to a periodic function or vice versa, and they may not be distinguishable. Therefore, we impose quantitative separation.

\begin{definition}
Let $f:\mathbb{R}\to\mathbb{R}$ be a differentiable function. We say $f$ is $(\Xi,d,\epsilon)$-periodic if $f$ is $\Xi$-periodic, and for any $t_{1},t_{2}\in\mathbb{R}$
with $\min_{n\in\mathbb{Z}}\left|t_{1}-t_{2}+n\Xi\right|\geq d$,
\[
\left|f(t_{1})-f(t_{2})\right|\geq\epsilon\text{ or }\left|f'(t_{1})-f'(t_{2})\right|\geq\epsilon.
\]
We say $f$ is $(d,\epsilon)$-non-periodic on an interval $I$ if for all $t_{1},t_{2}\in I$ with $\left|t_{1}-t_{2}\right|\geq d$,
\[
\left|f(t_{1})-f(t_{2})\right|\geq\epsilon\text{ or }
\left|f'(t_{1})-f'(t_{2})\right|\geq\epsilon.
\]
\end{definition}


The constant $d$ is the threshold that we ignore very self-nearness, and $\epsilon$ specifies the resolution at which two phase states are treated as separated. Larger $\epsilon$ gives stronger separation and a larger certified topology scale, while smaller $\epsilon$ allows a broader signal class but weakens the separation margin.

These separated classes intentionally exclude degenerate or mixed regimes. Constant and near-constant functions have ill-posed periods and degenerate delay clouds, while partial recurrence can create branching supports. See Figure \ref{fig:1} for illustrations of $\Xi$-periodic functions that are each $(\Xi,d,\epsilon)$-periodic (Fig \ref{fig:1a}) and not $(\Xi,d,\epsilon)$-periodic (Fig \ref{fig:1b}). We focus on globally periodic or globally non-periodic behavior, whose delay-space topology separates cleanly as shown in Section~\ref{sec:top_period}.



\begin{figure}[H]
    \centering
    \begin{subfigure}{0.48\linewidth}
        \centering
        \includegraphics[width=\linewidth]{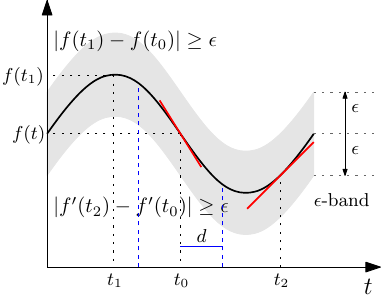}
        \caption{$(\Xi,d,\epsilon)$-periodic function}
        \label{fig:1a}
    \end{subfigure}
    \hfill
    \begin{subfigure}{0.45\linewidth}
        \centering
        \includegraphics[width=0.95\linewidth]{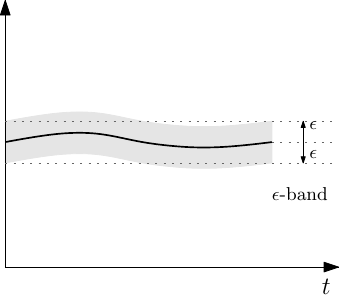}
        \caption{Near-constant function; $\Xi$-periodic\\but not $(\Xi,d,\epsilon)$-periodic function}
        \label{fig:1b}
    \end{subfigure}
    
    \caption{Example of $(\Xi,d,\epsilon)$-periodic function and not $(\Xi,d,\epsilon)$-periodic function}
    \label{fig:1}
\end{figure}
\begin{assumption}
\label{ass:periodic_either}
The true sampling function $f:\mathbb{R}\to\mathbb{R}$ is either $(\Xi,d,\epsilon)$-periodic for some $\Xi,d,\epsilon >0$, or $(d,\epsilon)$-non-periodic on $J_{m,\tau}$ for some $\epsilon >0$. In the periodic case, we assume the observation interval contains a full period, that is, $T_{\max} - T_{\min} \ge \Xi$.
\end{assumption}

We also impose regularity and scale-compatibility conditions, stated precisely in Appendix~\ref{sec:technical_assumptions}. In brief, $f$ is $C^{2,1}$ on the relevant domain, its first two derivatives are uniformly bounded, and the phase trajectory is prevented from locally flattening in both first and second derivative. The separation scale $d$ is required to be smaller than the curvature scale determined by these constants. These conditions, Assumptions~\ref{ass:fp_fpp_nonzero} and~\ref{ass:d inequality}, are used to certify positive reach of the sliding-window curve.


\medskip
In practice, instead of observing the exact time-delay embedded samples $\{SW_{m,\tau}f(t_j)\}$, we only have access to a finite sequence of discrete observations $f(s_0), f(s_1), \dots, f(s_N)$ subject to measurement error. To bound the approximation error of this empirical embedding, we assume the following conditions:

\begin{assumption}[Bounded-error linear interpolation]\label{ass:errorInterpolation}
Let 
$
\mathcal{G}_N = \{s_0 < s_1 < \cdots < s_N\}
$
be an interpolation grid covering $J_{m,\tau}$, and write
$
\Delta_N := \max_{0 \le i < N} (s_{i+1} - s_i)
$.

Instead of observing $f$ exactly on this grid, suppose we observe
\[
Y_i = f(s_i) + \xi_i, \qquad |\xi_i| \le \sigma_N, \quad 0 \le i \le N.
\]

Let $\hat{f}_N$ be the piecewise-linear interpolant of $(s_i, Y_i)_{i=0}^N$. 
We assume that
\[
|f''(t)| \le L_{\mathrm{int}} \qquad \text{for all } t \in J_{m,\tau}.
\]

This assumption is only used for the robust confidence and testing statements. 
The interpolation error may be deterministic or random; if the bound 
$|\xi_i| \le \sigma_N$ holds only on a high-probability event, the final 
confidence bound gains the probability of the complement of that event.
\end{assumption}

\section{Topology of Periodicity and non-Periodicity}
\label{sec:top_period}

In this section, we show that the sliding-window support is topologically trivial for non-periodic functions and circle-like for periodic functions. These support-level statements are the topological input for the testing procedure.


For a set $A \subset \mathbb{R}^{m+1}$ and $r>0$, let 
$A^{r}:=\{x:d(x,A)<r\},\; d(x,A)=\inf_{y\in A}\|x-y\|_{2}$
be the $r$-tubular neighborhood of $A$.
We first consider the non-periodic case.

\begin{theorem}

\label{thm:nonperiodic_contractible}

Suppose $f$ satisfies Assumption \ref{ass:periodic_either}, \ref{ass:fp_fpp_nonzero}, and \ref{ass:d inequality}, and is $(d,\epsilon)$-non-periodic on $J_{m,\tau}$. If
$$m\tau > \max \{\delta/(2L), \epsilon/(4L)\}, \quad \tau < \min \{\delta/(8L), \epsilon/(16L)\}$$
then for every $0 < r < a_{m,\tau}$ where
$$a_{m,\tau} := \sqrt{m+1}\min\Bigl\{\frac{\epsilon}{32}\min\big\{1,\frac{\epsilon}{16L}\big\},\frac{\nu_{\delta,L}^2}{2L}\Bigr\},$$
we have that $(SW_{m,\tau}f(\mathbb{T}))^{r}$ deformation retracts onto $SW_{m,\tau}f(\mathbb{T})$, which is homeomorphic to an interval. In particular, the offset is contractible.

\end{theorem}

\begin{corollary}

\label{cor:non-periodic_homology}

Under the assumptions of Theorem \ref{thm:nonperiodic_contractible}, let $\mathcal{P}_{1}(SW_{m,\tau}f)$ denote the $1$-dimensional persistence diagram associated with the \v{C}ech filtration on $SW_{m,\tau}f(\mathbb{T})$. Then
\[
\mathcal{P}_{1}(SW_{m,\tau}f(\mathbb{T}))\cap \big\{\left[0,a_{m,\tau} \right)\times\mathbb{R}\big\}=\emptyset.
\]

\end{corollary}

Next, we look at the periodic case.

\begin{theorem}

\label{thm:periodic_circle}

Suppose $f$ satisfies Assumption \ref{ass:periodic_either}, \ref{ass:fp_fpp_nonzero}, and \ref{ass:d inequality}, and is $(\Xi,d,\epsilon)$-periodic. If
$$m\tau > \max \{\delta/(2L), \epsilon/(4L)\}, \quad \tau < \min \{\delta/(8L), \epsilon/(16L)\}$$
then for every $0 < r < a_{m,\tau}$,
we have that $(SW_{m,\tau}f(\mathbb{T}))^{r}$ deformation retracts onto $SW_{m,\tau} f(\mathbb{T}))$, which is homeomorphic to $S^{1}$.

\end{theorem}

\begin{corollary}

\label{cor:periodic_homology}

Under the assumptions of Theorem \ref{thm:periodic_circle}, let $\mathcal{P}_{1}(SW_{m,\tau}f)$ denote the $1$-dimensional persistence diagram associated with the \v{C}ech filtration on $SW_{m,\tau}f(\mathbb{T})$. Then
\[
\mathcal{P}_{1}(SW_{m,\tau}f(\mathbb{T}))\cap \big\{\left[0,a_{m,\tau} \right)\times\mathbb{R}\big\}=\{(0,D)\},
\]
where $D \ge a_{m,\tau}$.
\end{corollary}

\section{Subsampling Confidence Bound}
\label{sec:SCBforSlidingWindow}

While Section \ref{sec:top_period} showed the underlying support possesses a trivial or circular topology, our actual observations are restricted to finite samples. Accordingly, this section proposes a methodology for deriving confidence bounds from sampled data to infer the true topological features of the support.

\subsection{Confidence bound }\label{sec:subsample_Confidence_bound}

Based on the observed sample $T = \{t_0,t_1,\cdots, t_N\} \subset \mathbb{T}$, we obtain the empirical sliding window embedding $X_{m,\tau} = \operatorname{SW}_{m,\tau} f(T)$, with its true underlying support denoted by $M_{m,\tau} = \operatorname{SW}_{m,\tau} f(\mathbb{T})$.

To construct confidence bounds, we adapt the subsampling strategy introduced by~\cite{Fasy_2014}. 
This choice is motivated by the fact that other common methods, such as the standard bootstrap, lack theoretical guarantees for Vietoris-Rips persistent homology
. Furthermore, subsampling has been demonstrated to be the most effective approach among discrete resampling techniques.

We first begin by outlining the algorithmic procedure, followed by a formal derivation of its theoretical guarantees.

Let $X_{m,\tau,b}^{(j)}$ denote the $j$-th subsample of size $b$ drawn from $X_{m,\tau}$. To estimate the distribution of the Hausdorff distance between the empirical sample and its subsamples, we compute the following empirical estimator over $B$ iterations:
$
\bar{L}_{m,\tau}(t) = B^{-1} \sum_{j=1}^{B} I \bigl( d_H ( X_{m,\tau}, X_{m,\tau,b}^{(j)} ) > t \bigr)
$.

The complete procedure is formalized in the following algorithm.

\begin{algorithm}
\caption{Implementation of Confidence bound}\label{alg:cap}
\textbf{Input:} $(t_0,Y_0),(t_1,Y_1),\cdots, (t_N,Y_N)$, m, n, b, $\tau$, $\alpha$.

1. Compute the sliding window embedding of the sample $X_{m,\tau} = SW_{m,\tau}f(\{t_0, t_1, \dots, t_N\})$. 

(Alternatively, under the error assumptions of Section \ref{sec:linearInterpolation}, $\hat{X}_{m,\tau} = SW_{m,\tau} \tilde{f}_N(\{t_0, t_1, \dots, t_N\})$)

2. Compute Monte Carlo subsamples $X_{m,\tau,b}^{(j)}$, for $j=1,\ldots,B$.

3. Compute upper $\alpha$ quantile of Hausdorff distances $\bar{L}_{m,\tau}^{-1}(\alpha)$.

\textbf{Output:} $c_\alpha=2\bar{L}_{m,\tau}^{-1}(\alpha)$
\end{algorithm}

Our subsequent theorem demonstrates the construction of valid confidence bounds for the homological features of the support based on the empirical sample. To this end, we define the theoretical tail probability function, computed over all $\binom{n}{b}$ possible subsamples of size $b$:
$$
L_{m,\tau}(t) = \frac{1}{\binom{n}{b}} \sum_{j=1}^{\binom{n}{b}} I \left( d_H ( X_{m,\tau}, X_{m,\tau,b}^{(j)} ) > t \right).$$

For this monotonically non-increasing tail function, we adopt the standard convention for its generalized inverse:
\[
L_{m,\tau}^{-1} (\alpha) := \inf \{t \ge 0 : L_{m,\tau}(t) \le \alpha \}.
\]

\begin{theorem}
\label{thm:nonStandconfidence}
Fix embedding parameters $m, \tau$. Suppose the sampling distribution satisfies Assumption \ref{ass:Sampling_Density}. Suppose $f$ satisfies Assumption \ref{ass:periodic_either}, \ref{ass:fp_fpp_nonzero}, and \ref{ass:d inequality}. Suppose $m, \tau$ satisfies
\[
m\tau > \max \{\delta/(2L), \epsilon/(4L)\}, \quad \tau < \min \{\delta/(8L), \epsilon/(16L)\}.
\]
Let $b = o(n/\log{n})$ be a sequence tending to infinity as $n \to \infty$, and define $c_{\alpha} = 2(L_{m,\tau})^{-1} (\alpha)$. Then,
    \[
        P\Bigl(d_B(\mathcal{P}(X_{m,\tau}), \mathcal{P}(M_{m,\tau})) > c_{\alpha} \Bigr) \le \alpha + O\big((b/n)^{1/4}\big).
    \]
\end{theorem}




\begin{lemma}
\label{lem:calpha_converge_zero}
Under the assumptions of Theorem \ref{thm:nonStandconfidence}, let $m, \tau$ and $\alpha \in (0,1)$ be fixed. If $b\rightarrow \infty$ and $b \le n$, then
$$c_{\alpha} \xrightarrow[]{P} 0.$$
\end{lemma}

\subsection{Interpolation Error}\label{sec:linearInterpolation}

For finite data, the exact value $f(t + jr)$ are usually not in the data, so we used linear interpolation from noisy scalar observations. For the linear interpolation $\hat{f}_N$, we can say the error term as
\[
\eta_N := \sigma_N +  \frac{\Delta_N^2}{8}L_\text{int}, 
\qquad
\zeta_{m,N} := \sqrt{m+1}\,\eta_N.
\]
Let $\hat{X}_{m,\tau} := SW_{m,\tau}\,\tilde{f}_N(T)$ be the observed sliding-window cloud obtained from the interpolation. For all size $b$ subsamples of $\hat{X}_{m,\tau}$, define
$
\hat{L}_{m,\tau}(t) = \frac{1}{\binom{n}{b}} \sum_{j=1}^{\binom{n}{b}} 
I\big( d_H(\hat{X}_{m,\tau}, \hat{X}^{(j)}_{m,\tau,b}) > t \big)
$.

We can set $\hat{c}_\alpha := 2\hat{L}^{-1}_{m,\tau}(\alpha)$ and $\hat{c}_\alpha^{\mathrm{err}} := \hat{c}_\alpha + 5\zeta_{m,N}$, where $\hat{L}^{-1}_{m,\tau}$ uses the same generalized inverse convention as $L^{-1}_{m,\tau}$.

\begin{theorem}[Confidence bound with bounded interpolation error]\label{thm:confidenceBoundInterpolation}
Assume the conditions of Theorem \ref{thm:nonStandconfidence} and Assumption \ref{ass:errorInterpolation}. Then,
\[
P\Big( d_B(\mathcal{P}(\hat{X}_{m,\tau}), \mathcal{P}(M_{m,\tau})) > \hat{c}_\alpha^{\mathrm{err}} \Big)
\leq \alpha + O \big(\left(b/n\right)^{1/4}\big).
\]
\end{theorem}

\begin{corollary}\label{cor:limitInterpolationError}
    Under the assumptions of Theorem \ref{thm:confidenceBoundInterpolation}, fix $m, \tau$ and let $\alpha \in (0,1)$ be fixed. If $b \to \infty$, $b \le n$, and $\zeta_{m,N} \to 0$, then
    \[
    \hat{c}_\alpha^{\mathrm{err}} \xrightarrow{P} 0.
    \]
\end{corollary}

\section{Topological Periodicity Test (TopPT)}
\label{sec:testing}

The quantitatively periodic and non-periodic classes are disjoint but not
complements, so we restrict the hypotheses to these two separated regimes. For
fixed embedding parameters $m,\tau$,
\[
H_0:f\text{ is }(d,\epsilon)\text{-non-periodic on }J_{m,\tau},
\qquad
H_1:f\text{ is }(\Xi,d,\epsilon)\text{-periodic for some }\Xi>0.
\]
At the support level, Corollary~\ref{cor:periodic_homology} gives a genuine
1-cycle born at zero and persisting at least until $a_{m,\tau}$ under $H_1$,
whereas Corollary~\ref{cor:non-periodic_homology} gives no 1-dimensional
feature born before $a_{m,\tau}$ under $H_0$. Thus a confidence radius smaller
than $a_{m,\tau}/4$ separates the two cases.

Motivated by this topological separation, we introduce the Topological Periodicity Test (TopPT), which provides a statistical test for periodicity based on persistent homology.
\begin{definition}[Topological Periodicity Test (TopPT)]\label{def:periodicityTesting}
    Let $Y$ denote an observed delay point cloud and  $r_\alpha$ a generic confidence radius. Define rejection region and associated separation event
    \[
    R(r_\alpha):=\{(b,d)\in\mathbb{R}^2:b\le r_\alpha,\ d-b>2r_\alpha\},
    \qquad
    A_\alpha^r:=\{a_{m,\tau}>4r_\alpha\}.
    \]
    Using the boundary convention that assigns equality in birth to the rejection region and equality in persistence to its complement, 
    we define Topological Periodicity Test (TopPT) as follows:
    \[
    \psi(Y)= 0 \:\text{ if } \mathcal{P}(Y)\cap R(r_\alpha)=\emptyset,\qquad \psi(Y)= 1 \:\text{if } \mathcal{P}(Y)\cap R(r_\alpha)\neq\emptyset.
    \]
    The $\psi(Y)=1$ corresponds to rejection of $H_0$ and existence of persistent periodic 1-cycle.
\end{definition}


For exact observations, we apply Definition~\ref{def:periodicityTesting} with
$(Y,r_\alpha,\psi_{m,\tau,r_\alpha})=(X_{m,\tau},c_\alpha,\phi)$. For bounded-error interpolation, we instead use $(Y,r_\alpha,\psi_{m,\tau,r_\alpha})=(\hat X_{m,\tau},\hat c_\alpha^{\mathrm{err}},\hat\phi)$, where $\hat c_\alpha^{\mathrm{err}}$ is defined in Section~\ref{sec:linearInterpolation}. This yields the testing rule used in the bounded interpolation setting.\label{sec:testing_interpolation}


The following image illustrates the persistence diagram of $X_{m,\tau}$ with a true  persistence point $(0,d)$ under $H_1$.

\begin{figure}[H]
    \centering
    \includegraphics[width=0.5\linewidth]{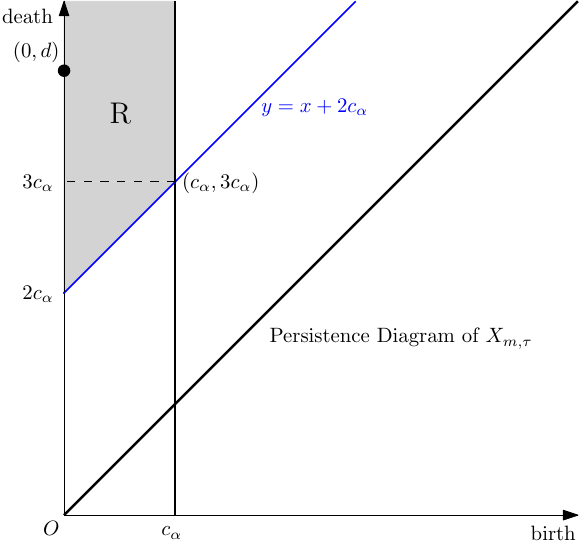}
    \caption{Diagram of Hypothesis Testing}
    \label{fig:placeholder}
\end{figure}

\begin{theorem}[Testing error bound]
\label{thm:testing_error_bound}
\label{thm:type1err}
\label{thm:type2err}
\label{thm:type1err_interpolation}
\label{thm:type2err_interpolation}
Assume either the exact-observation confidence bound of
Theorem~\ref{thm:nonStandconfidence} with
$(Y,r_\alpha)=(X_{m,\tau},c_\alpha)$, or the bounded-interpolation confidence
bound of Theorem~\ref{thm:confidenceBoundInterpolation} with
$(Y,r_\alpha)=(\hat X_{m,\tau},\hat c_\alpha^{\mathrm{err}})$. Then, under
$H_0$,
\[
P(\psi(Y)=1,A_\alpha^r\mid H_0)
\le \alpha+O\big(\left(b/n\right)^{1/4}\big),
\]
and, under $H_1$,
\[
P(\psi(Y)=0,A_\alpha^r\mid H_1)
\le \alpha+O\big(\left(b/n\right)^{1/4}\big).
\]
Consequently,
\[
P(\psi(Y)=1\mid H_0)
\le
\alpha+P((A_\alpha^r)^c\mid H_0)
+O\big(\left(b/n\right)^{1/4}\big),
\]
and
\[
P(\psi(Y)=0\mid H_1)
\le
\alpha+P((A_\alpha^r)^c\mid H_1)
+O\big(\left(b/n\right)^{1/4}\big).
\]
\end{theorem}

\begin{corollary}[Asymptotic testing]
\label{cor:testing_asymptotic}
\label{cor:testing_interpolation_asymptotic}
For fixed $m,\tau$, let $b\to\infty$ with $b=o(n/\log n)$. The exact rule
satisfies
\[
P(\phi(X_{m,\tau})=1\mid H_0)\le \alpha+o(1),
\qquad
P(\phi(X_{m,\tau})=0\mid H_1)\le \alpha+o(1).
\]
The same conclusion holds for the bounded-interpolation rule whenever
$\zeta_{m,N}\to0$, with $\phi(X_{m,\tau})$ replaced by
$\hat\phi(\hat X_{m,\tau})$.
\end{corollary}

\subsection{Finite parameter searches}

The guarantees above are stated for fixed $m,\tau$. For a fixed finite grid
$\mathcal{G}=\{(m_l,\tau_l):1\le l\le K\}$, run the exact or bounded-error
rule at level $\alpha/K$ for each candidate and let $\psi_l$ and
$A_{\alpha/K,l}^r$ denote the corresponding decision and separation event.
Define
\[
\Phi_{\mathcal{G}}=1
\Longleftrightarrow
\psi_l=1\text{ for at least one }1\le l\le K.
\]

\begin{corollary}
\label{cor:Finite_grid_testing}
Assume $K$ is fixed. If, under $H_0$, every candidate in $\mathcal{G}$ satisfies
the non-periodic assumptions of the corresponding confidence theorem, then
\[
P\bigg(\Phi_{\mathcal{G}}=1,\bigcap_{l=1}^K A_{\alpha/K,l}^r\mid H_0\bigg)
\le
\alpha+O\big(K\left(b/n\right)^{1/4}\big).
\]
If, under $H_1$, at least one candidate $l_*$ satisfies the periodic
assumptions, then
\[
P\bigg(\Phi_{\mathcal{G}}=0,A_{\alpha/K,l_*}^r\mid H_1\bigg)
\le
\alpha/K+O\big(\left(b/n\right)^{1/4}\big).
\]
Consequently, if the relevant separation events have probability tending to
one, the finite-grid test controls both error probabilities at level
$\alpha+o(1)$.
\end{corollary}

\section{Experiments}
\label{sec:simul}

We evaluate the proposed periodicity test TopPT on synthetic signals and respiratory
waveform data. All TopPT experiments use raw sliding-window embeddings, and
persistence diagrams are computed with Vietoris--Rips complexes as a
computational surrogate for the distance, equivalently \v{C}ech, persistence
used in the theory. Full signal definitions, preprocessing steps, parameter
choices, and reproducibility details are given in Appendix~\ref{sec:err_exp}.
For scalar baselines, we use the Generalized Lomb--Scargle periodogram
(GLS)~\cite{Zechmeister2009}, Fisher's $g$-test~\cite{Fisher1929Harmonic}, and
a permutation max-periodogram test, all run on the same scalar observations
used to build the delay clouds.

\subsection{Synthetic Data}

We test two periodic signals with known period $\Xi=2\pi/5$ and four
non-periodic signals with trend-like or low-frequency behavior. The goal is to
check whether TopPT detects genuine delay-space loops without rejecting
structured but non-periodic examples. We use the known-period delay choice
$\tau=\Xi/(m+1)$ and compare raw TopPT, robust TopPT, GLS, Fisher's $g$-test, and a
permutation max-periodogram test. Signal definitions, noise models,
interpolation details, and parameter choices are in
Appendix~\ref{sec:synthetic_design}.

\begin{table}[!t]
\centering
\caption{Synthetic rejection rates averaged over bounded-error settings. For
periodic signals, entries are detection rates; for non-periodic signals, entries
are false rejection rates.}
\label{tab:synthetic_main_summary}
\small
\begin{tabular}{llccccc}
\toprule
Signal & class & raw TopPT & robust TopPT & Lomb--Scargle & Fisher $g$ & permutation \\
\midrule
$f_1$ & periodic & 1.000 & 0.699 & 1.000 & 1.000 & 1.000 \\
$f_2$ & periodic & 1.000 & 0.457 & 1.000 & 1.000 & 1.000 \\
$h_1$ & non-periodic & 0.000 & 0.000 & 1.000 & 1.000 & 1.000 \\
$h_2$ & non-periodic & 0.000 & 0.000 & 1.000 & 1.000 & 1.000 \\
$h_3$ & non-periodic & 0.000 & 0.000 & 1.000 & 1.000 & 1.000 \\
$h_4$ & non-periodic & 0.000 & 0.000 & 1.000 & 1.000 & 1.000 \\
\bottomrule
\end{tabular}
\end{table}

Table~\ref{tab:synthetic_main_summary} shows that raw TopPT detects both
periodic signals and does not reject the four structured non-periodic examples.
Robust TopPT is more conservative: it keeps the same zero false
rejection rate but has lower power because the confidence radius includes the
interpolation-error correction. The scalar periodogram baselines reject all
signals in this experiment, including the structured non-periodic examples,
which suggests that they can flag one-dimensional patterns that do not
necessarily form a statistically significant loop in the delay embedding.

\subsection{Real Data}

We evaluate local respiratory periodicity on the Fantasia database and the
BIDMC PPG and Respiration dataset from PhysioNet~\cite{Goldberger2000PhysioNet}.
Fantasia contains 40 healthy resting-subject recordings with respiration
channels~\cite{FantasiaPhysioNet,Iyengar1996Fantasia}, and BIDMC contains 53
critical-care recordings with impedance respiratory signals~\cite{BIDMCPhysioNet,
Pimentel2016RespRate}. Each record is tested over local windows, with candidate
periods estimated from autocorrelation and Welch spectral peaks and delay
$\tau=\mathrm{period}/(m+1)$. We apply a Bonferroni correction over the realized
finite window-period search set; full preprocessing and search-grid details are
in Appendix~\ref{sec:real_protocol}.

\begin{table}[!t]
\centering
\caption{Real-data rejection counts. TopPT columns report the number of rejected
records and local windows. Baseline entries are written as record-level
rejections followed by window-level rejections.}
\label{tab:real_summary}
\small
\setlength{\tabcolsep}{4pt}
\begin{tabular}{lcccccc}
\toprule
Dataset & records & windows & TopPT records & TopPT windows & Lomb--Scargle & Fisher $g$/permutation \\
\midrule
Fantasia & 40 & 280 & 35/40 & 142/280 & 40/40, 280/280 & 40/40, 280/280 \\
BIDMC    & 53 & 318 & 51/53 & 249/318 & 53/53, 318/318 & 53/53, 318/318 \\
\bottomrule
\end{tabular}
\end{table}

Table~\ref{tab:real_summary} shows that TopPT rejects for most records
while rejecting only a subset of local windows: 35 of 40 Fantasia records and
51 of 53 BIDMC records. The scalar periodogram baselines reject every window in
both datasets. Since these respiratory records are expected to contain cyclic
components and do not come with null labels, this experiment should be read as
a practical detection and localization study; the synthetic null experiments
above are the main evaluation of false rejections. The difference between TopPT
and the scalar baselines suggests that TopPT is more selective: it
rejects when the local delay embedding contains a persistent loop beyond the
subsampling confidence radius, rather than whenever the scalar time series has
strong oscillatory content. We report only raw TopPT on real data, since the
robust bounded-error rule requires certified measurement-error and
interpolation-smoothness bounds that the recordings do not provide.

\section{Concluding Remarks}\label{sec:remark}

We developed a rigorous statistical framework TopPT for testing periodicity by analyzing the topology of time-delay embeddings. Our approach leverages the topological separation between the sliding-window supports of periodic ($S^1$-like) and non-periodic (interval-like) signals. To achieve this, we establish a quantitative bound, which allows us to precisely locate these features as points on the persistence diagram. By combining this with a subsampling confidence radius for persistence diagrams, we constructed a hypothesis test with asymptotic error control, which we also extended to interpolated scalar observations. The multiple testing framework for candidate window sizes and dimensions is also provided. Empirical results on synthetic and real data demonstrate that our approach does not simply evaluate whether a scalar time series contains oscillatory energy, but rigorously certifies the presence of a coherent persistent loop beyond the subsampling confidence radius. Thus, non-rejection does not imply an absence of oscillation, but rather the absence of a statistically certified loop at the tested window, period, and embedding scale. This provides a useful intermediate description between fully periodic and non-periodic behavior, especially for real signals whose rhythm may drift in amplitude, phase, or waveform shape.

Several directions remain open. Moving beyond strictly separated classes to encompass partial recurrence, quasi-periodic, or almost-periodic dynamics would significantly broaden the method's real-world applicability. Additionally, establishing error guarantees for adaptive period selection is important since candidate periods for real data are naturally estimated from the same observations being tested.

\bibliographystyle{plain}
\bibliography{reference_}

\appendix

\section*{Impact Statement}
This paper presents work whose goal is to advance the field of machine learning. There are many potential societal consequences of our work, none of which we feel must be specifically highlighted here.

\section{Technical Assumptions}
\label{sec:technical_assumptions}

This appendix records the regularity and scale-compatibility assumptions used
to control the reach of the sliding-window curve in Section~\ref{sec:top_period}.

\begin{assumption}[Sliding-window regularity and nondegeneracy]
\label{ass:fp_fpp_nonzero}
The true sampling function $f:\mathbb{R}\to\mathbb{R}$ is $C^{2,1}$ on the
relevant domain, meaning that $f$ is $C^2$ and $f''$ is Lipschitz. There exists
some $\delta, L>0$ such that
$\left|f'(t)\right|\geq\delta$ or $\left|f''(t)\right|\geq\delta$
for every relevant $t$, and
$\left| f'(t) \right|, \left|f''(t)\right|\leq L,\;
\left| f''(s)-f''(t) \right| \le L\left| s-t \right|$
for every relevant $s, t$.
\end{assumption}

Define
\[
\nu_{\delta, L}
:=
\frac{\delta}{16}\min\left\{1,\frac{\delta}{8L}\right\},
\qquad
d_{\delta, L}
:=
\frac{\pi}{2L^2}\nu_{\delta,L}^2.
\]

\begin{assumption}[Scale compatibility]
\label{ass:d inequality}
The separation scale $d$ in Assumption~\ref{ass:periodic_either} satisfies
$d < d_{\delta, L}$.
\end{assumption}

Assumption~\ref{ass:fp_fpp_nonzero} prevents the sliding-window trajectory from
locally collapsing, while Assumption~\ref{ass:d inequality} ensures that the
global separation condition activates before curvature alone can allow a near
self-intersection.

\section{Simplicial complex, Filtration, and Persistent homology}
\label{sec:simplicial_complex}

Given a set $V$, an \textit{(abstract) simplicial complex}
is a set $K$ of subsets of $V$ such that $\alpha\in K$ and
$\beta\subset\alpha$ implies $\beta\in K$. Each set $\alpha\in K$
is called its \textit{simplex}. The \textit{dimension} of a simplex
$\alpha$ is $\dim\alpha=\text{card }\alpha-1$, and the dimension
of the simplicial complex is the maximum dimension of any of its simplices.

Assume we observe finite samples $X$ from a metric space $(\mathbb{X},d)$. 
\textit{\v{C}ech complex} is one method for approximating the topology of the underlying space $\mathbb{X}$ by observed samples $X$.

\begin{definition}[\v{C}ech complex] \label{def:cech} Let
	$X\subset\mathbb{X}$ be finite and $r>0$.
	The \v{C}ech complex is the simplicial complex 	
    \[\mathcal{C}(r)\coloneqq\{ \sigma\subset X:\ \cap_{x\in\sigma}B_{\mathbb{X}}(x,r)\neq\emptyset\}.
	\]
\end{definition}

Another common choice is the \emph{Vietoris-Rips complex}, also referred to as \emph{Rips complex}, where simplexes are built based on pairwise distances among its vertices.

\begin{definition}[Vietoris-Rips complex] \label{def:rips}
    Let	$X\subset\mathbb{X}$ be finite and $r>0$.	
	The \emph{Vietoris-Rips complex} $\mathcal{VR}(r)$ is the simplicial complex defined as
	\[
	\mathcal{VR}(r)\coloneqq\{ \sigma\subset X:d(x_{i},x_{j})<2r,\forall x_{i},x_{j}\in\sigma\} .
	\]
\end{definition}

Note that \v{C}ech complex and Vietoris-Rips complex have the following interleaving
inclusion relationship 
\[
\mathcal{C}(r)\subset \mathcal{VR}(r)\subset\mathcal{C}(2r).
\]
In particular, when $\mathbb{X}\subset\mathbb{R}^{d}$
is a subset of a Euclidean space of dimension $d$, then the constant $2$ can be tightened to $\sqrt{\frac{2d}{d+1}}$ (e.g., see Theorem 2.5 in \cite{deSilvaG2007}):
\[
\mathcal{C}(r)\subset \mathcal{VR}(r)\subset\mathcal{C}\left(\sqrt{\frac{2d}{d+1}}r\right).
\]



A \emph{filtration} $\mathcal{F}$ is a collection of subcomplexes, formally defined as follows. 
\begin{definition}[Filtration]
	A \emph{filtration} $\mathcal{F}=\{ K_{a}\subset K\}_{a\in\mathbb{R}}$
	is a collection of subcomplexes of $K$ such that $a\leq b$ implies that $K_{a}\subset K_{b}$.
\end{definition}

By defining a filtration, we can approximate a point cloud at different levels of resolution. \emph{Persistent homology} is the subsequent multiscale tool used to capture the topological features of the resulting complex $K$. For any filtration $\mathcal{F}$ and non-negative integer $k$, the persistent homology $PH_{k}\mathcal{F}$ is simply an ordered sequence of $k$-dimensional homology groups—one for every step in the filtration $\mathcal{F}$.

\begin{definition}[Persistent homology] 
	
	Let $\mathcal{F}$ be a filtration and let $k\in\mathbb{N}_{0}$.
	The associated $k$-th \emph{persistent homology} $PH_{k}\mathcal{F}$
	is a collection of groups $\{ H_{k}(K_{a})\} _{a\in\mathbb{R}}$ of each subcomplex $K_{a}$ in $\mathcal{F}$ 
	equipped with homomorphisms $\{ \imath_{k}^{a,b}\} _{a\leq b}$,
	where $H_{k}(K_{a})$ is the $k$-th dimensional homology group of
	$K_{a}$ and $\imath_{k}^{a,b}\colon H_{k}K_{a}\to H_{k}K_{b}$ is the homomorphism induced by the inclusion map $\imath:K_{a}\hookrightarrow K_{b}$. If we write $H_k^{a,b} :=\text{im}(\imath_k^{a,b})$, the corresponding $k$-th \emph{persistent Betti numbers} are the ranks of these groups, $\beta_k^{a} = \text{rank }H_k^{a}$.
	
\end{definition}

For the $k$-th persistent homology
$PH_{k}\mathcal{F}$, the set of filtration levels at which a specific
homology appears is always an interval $[b,d)\subset[-\infty,\infty]$. We often say that the homological feature $\alpha$ is born at $b$ and dies at $d$. To be more precise, let $\gamma$ be a class in $H_k(\mathcal{F}_b)$. We say it is born at $b\in \mathbb{R}$ if $\gamma\notin H_k^{a,b}$ for all $a<b$. Furthermore, if $\gamma$ is born at $b$ then it dies entering $d$ if it merges with an older class as we go from $\mathcal{F}_c$ to $\mathcal{F}_d$ for any $c \in [b,d)$.

By visualizing these pairs as points in the plane, one obtains the \emph{persistence diagram} as below.

\begin{definition}[Persistence diagram] 
	Let $\mathbb{R}^2_\ast \coloneqq \{ (b,d) \in (\mathbb{R}\cup\infty)^2: d > b \}$.
	Let $\mathcal{F}$ be a filtration and let $k\in\mathbb{N}_{0}$.
	The corresponding $k$-th persistence diagram $\mathcal{P}_{k}(\mathcal{F})$ is a
	finite multiset of $\mathbb{R}^2_\ast$, consisting
	of all pairs $(b,d)$, where $[b,d)$ is the interval of filtration
	values for which a specific homology class appears in $PH_{k}\mathcal{F}$. $b$
	is called a birth time and $d$ is called a death time.
	
\end{definition}

When topological information of the underlying space is approximated
by the observed points, it is often needed to compare two sets with
respect to their metric structures. We first endow the space of persistence diagrams with a metric.

\begin{definition}[Bottleneck distance]
    The bottleneck distance between two persistence diagrams $\mathcal{P}_k(\mathcal{F})$ and $\mathcal{P}_k(\mathcal{G})$ is defined by
    $$d_B(\mathcal{P}_k(\mathcal{F}),\mathcal{P}_k(\mathcal{G})) = \inf_{\gamma\in \Gamma } \sup_{p\in \mathcal{P}_k(\mathcal{F})} \Vert p-\gamma(p)\Vert_{\infty},$$
    where the set $\Gamma$ consists of all the bijections $\gamma:\mathcal{P}_k(\mathcal{F})\cup \Delta \rightarrow \mathcal{P}_k( \mathcal{G})\cup \Delta$, and $\Delta$ is the diagonal $\{(x,x):x\in \mathbb{R}\}\subset \mathbb{R}^2$ with infinite multiplicity.
\end{definition}

The \emph{Hausdorff distance} \citep[][Definition 7.3.1]{burago2001} is defined on sets embedded in the same metric
spaces. This distance measures how two sets are close to each other
in the embedded metric space. When $S\subset\mathbb{X}$, we denote
by $S^r$ the $r$-neighborhood of a set $S$ in a metric space,
i.e. $S^r=\bigcup_{x\in S}B_{\mathbb{X}}(x,r)$.

\begin{definition}[Hausdorff distance] 
	
	\label{def:distance_hausdorff}
	
	Let $\mathbb{X}$ be a metric space, and $X,Y\subset\mathbb{X}$
	be a subset. The \emph{Hausdorff distance} between $X$ and $Y$, denoted
	by $d_{H}(X,Y)$, is defined as 
	\[
	d_{H}(X,Y)=\inf\{r>0:\,X\subset Y^r\text{ and }Y\subset X^r\}.
	\]

\end{definition}

Finally, one of the most crucial properties of persistent homology for practical data analysis is its robustness to noise and minor perturbations in the data. This robustness is formally guaranteed by the Stability Theorem. In essence, the theorem states that if two point clouds are geometrically close to each other, their corresponding topological signatures—represented as persistence diagrams—will also be strictly close. 

\begin{theorem}[Stability of Persistence Diagrams~\cite{chazal2012_PD_Stabiility}]
Let $X$ and $Y$ be finite subsets of a metric space. Let $\mathcal{P}_k(X)$ and $\mathcal{P}_k(Y)$ denote the $k$-dimensional persistence diagrams of the \v{C}ech or Vietoris-Rips filtrations built on $X$ and $Y$, respectively. Then, the bottleneck distance between the two persistence diagrams is bounded by the Hausdorff distance between the original sets:
$$d_B(\mathcal{P}_k(X), \mathcal{P}_k(Y)) \leq  d_H(X, Y)$$
\end{theorem}

\section{Reach}
\label{sec:reach}

The reach of a set, introduced by~\cite{federer1959curvature}, is a fundamental measure of its geometric regularity. For a closed subset $A \subset \mathbb{R}^d$, its medial axis, denoted by $\text{Med}(A)$, is defined as the set of points in $\mathbb{R}^d$ that have at least two distinct nearest neighbors in $A$:
$$\text{Med}(A) = \left\{x\in\mathbb{R}^d \mid \exists q_1 \neq q_2 \in A, \Vert q_1-x\Vert = \Vert q_2 - x\Vert = d(x,A)\right\}$$
where $d(x,A) = \inf_{q \in A} \Vert q-x \Vert$. 

The reach of $A$ is then defined as the minimum distance between $A$ and its medial axis.

\begin{definition}[\cite{federer1959curvature}, Definition 4.1]
    The reach of a closed subset $A \subset \mathbb{R}^d$ is defined as
    $$\tau_A = \inf_{q\in A} d(q, \text{Med}(A)).$$
\end{definition}

\begin{figure}[H]
    \centering
    \includegraphics[width=0.5\linewidth]{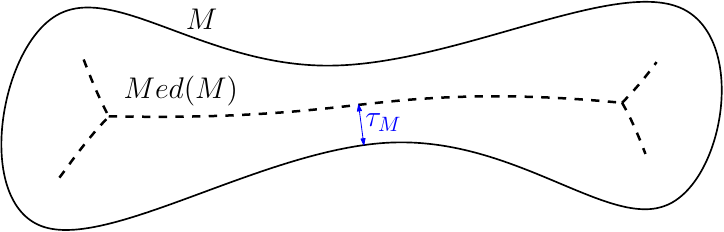}
    \caption{The medial axis and the reach~\cite{aamari2019estimatingreachmanifold}}
    \label{fig:placeholder}
\end{figure}

A strictly positive reach serves as a standard regularity assumption in various geometric and topological contexts. In particular, it is a crucial prerequisite for homological inference and manifold reconstruction; for instance, the $r$-tubular neighborhood $A^r$ deformation retracts to $A$ for any $r < \tau_A$. In light of this, Section \ref{sec:top_period} is essentially dedicated to verifying the lower bound for the reach of the time-delay embedded image.
\section{Statement on Section \ref{sec:Confidence_bound_Preliminary}}
\label{sec:confidence_subsampling_formal}

In Section \ref{sec:Confidence_bound_Preliminary} we addressed the results from \cite{Fasy_2014}. We describe the formal statement here. Assume the observed sample $\{X_1, \cdots, X_n\}$ from a distribution $P$ concentrated on the set $\mathbb{M}$. \cite{Fasy_2014} assumed following assumptions.

Assumption A1. $\mathbb{M}$ is a $d$-dimensional compact manifold without boundary, embedded in $\mathbb{R}^D$ and $\mathrm{reach}(\mathbb{M}) > 0$

Assumption A2. For each $x\in\mathbb{M}$, $\rho(x,t) = \frac{P(B(x,t/2))}{t^d}$ is bounded continuous function of $t$, differentiable for $t\in (0,t_0)$ and right differentiable at zero. $\rho(x,t)$ is bounded above from zero and infinity and there exists $t_0 > 0$ and some $C$ that
\[\sup_{x}\sup_{0\le t\le t_0} \bigg\vert \frac{\partial\rho(x,t)}{\partial t}\bigg\vert \le C < \infty.\]

Then, it is possible to construct the subsample bound from the Hausdorff distance between the subsample and the original sample. Define $L(t) = \frac{1}{\binom{n}{b}} \sum_{j=1}^{\binom{n}{b}} I(d_H(S_{b,n}^{j},S_n)>t)$, which is the summation of the indicator function over all possible $\binom{n}{b}$ subsamples of size $b$. As remarked in \cite{Fasy_2014}, in practice we shall use the Monte Carlo method to approximate this function. Denote the persistence diagram of $\{X_1, \cdots, X_n \}$ by $\hat{\mathcal{P}}$ and the persistence diagram of $\mathbb{M}$ by $\mathcal{P}$. For $b$ increasing to infinity as $n\rightarrow \infty$ and satisfying $b = o(\frac{n}{\log n})$,

\begin{theorem}[\cite{Fasy_2014} Theorem 3]
\label{thm:FasySubsampling}
   For all large $n$,
   \[\mathbb{P}(d_B(\hat{\mathcal{P}},\mathcal{P}) > c_{\alpha}) \le \alpha + O\bigg(\frac{b}{n}\bigg)^{1/4}.\]
\end{theorem}

For the non-periodic sliding-window support in Section~\ref{sec:top_period},
the support is an embedded interval rather than a manifold without boundary.
We use the following one-dimensional version of the same subsampling argument.
The point is that the ``without boundary'' assumption in \cite{Fasy_2014} is
used in Method I, Section~4.1, to obtain uniform covering and local-mass
estimates for the Hausdorff-distance subsampling argument. In our setting the
same estimates remain valid for compact embedded intervals with positive
reach. Indeed, if $M=\Gamma([a,b])$ is parametrized by a regular $C^2$ curve
and the sampling density along the parameter interval is bounded above and
away from zero, then for an interior point
\[
P(B(\Gamma(s),t/2))
=
\frac{q(s)}{\|\Gamma'(s)\|}t+O(t^2),
\]
whereas at an endpoint the local mass is one-sided and
\[
P(B(\Gamma(s),t/2))
=
\frac{q(s)}{2\|\Gamma'(s)\|}t+O(t^2).
\]
The factor $1/2$ changes only constants, not the order or differentiability of
the local mass function. Thus the Method I proof of \cite{Fasy_2014} applies
with modified constants once the following local mass condition is imposed.

\begin{proposition}[Subsampling for embedded curves, possibly with boundary]
\label{prop:curve_boundary_subsampling}
Let $M\subset\mathbb{R}^D$ be a compact embedded $C^2$ curve, either a circle
or an interval, with positive reach. Let $P$ be a probability distribution
supported on $M$. Suppose that, for
\[
\rho(x,t):=\frac{P(B(x,t/2))}{t},\qquad x\in M,
\]
there exists $t_0>0$ such that $\rho(x,t)$ is bounded above and away from
zero uniformly over $x\in M$ and $0<t<t_0$, is differentiable in $t$ for
$0<t<t_0$, and is right differentiable at $t=0$, with
\[
\sup_{x\in M}\sup_{0\le t\le t_0}
\left|\frac{\partial}{\partial t}\rho(x,t)\right|<\infty.
\]
Let $S_n=\{X_1,\ldots,X_n\}$ be i.i.d. from $P$. For all size-$b$ subsamples
$S_{b,n}^{(j)}$ of $S_n$, define
\[
L_b(t)
=\frac{1}{\binom{n}{b}}\sum_j
I\{d_H(S_{b,n}^{(j)},S_n)>t\},
\qquad
c_\alpha=2L_b^{-1}(\alpha).
\]
If $b\to\infty$ and $b=o(n/\log n)$, then, for all large $n$,
\[
\mathbb{P}\left(
d_B(\mathcal{P}(S_n),\mathcal{P}(M))>c_\alpha
\right)
\le
\alpha+O\left(\frac{b}{n}\right)^{1/4}.
\]
\end{proposition}

\section{Proofs for Section~\ref{sec:top_period}}

\begin{claim}
\label{claim:time_delay_norm_lower}
Let $h:I\supset[u,u+m\tau]\to\mathbb{R}^{m+1}$ be $C^{1}$. Suppose
that for all $s\in I$, 
\[
|h(s)|\ge\eta\qquad\text{or}\qquad|h'(s)|\ge\eta,
\]
and for all $r,s\in I$, 
\[
|h'(s)|\le L\qquad\text{and}\qquad|h'(s)-h'(r)|\le L|s-r|.
\]
If
\[
m\tau\ge\frac{\eta}{2L},\qquad\tau\le\frac{\eta}{8L},
\]
then 
\[
\left\Vert SW_{m,\tau}h(u)\right\Vert _{2}\ge\sqrt{m+1}\,\frac{\eta}{16}\min\left\{ 1,\frac{\eta}{8L}\right\} .
\]
\end{claim}

\begin{proof}
Let $\ell=\eta/(8L)$ and write the grid as 
\[
u_{j}=u+j\tau,\quad0\le j\le m.
\]
We first prove the bound on any consecutive
block of grid points whose time span is between $\eta/(2L)$ and $2\eta/L$.
Let $I_B=[u_p,u_q]$ be such a block and let $c=(u_p+u_q)/2$.

If $|h(c)|\ge\eta$, then $|h'|\le L$ gives $\vert h(s) \vert \ge \eta/2$ whenever $\vert s-c\vert \le \eta/ (2L)$. Since the block has half-span
at least $\eta/(4L)$, the intersection of this interval with $I_B$ contains
an interval of length at least $\ell$.

If $|h'(c)|\ge\eta$, then the Lipschitz bound on $h'$ implies that $h'$ has
constant sign and 
\[
|h'(s)|\ge\frac{\eta}{2}\qquad\text{whenever }\left|s-c\right|\leq\frac{\eta}{2L}.
\]
If $h(c)h'(c)\ge0$ we take
$[c+\ell,c+2\ell]$, while if $h(c)h'(c)<0$ we take
$[c-2\ell,c-\ell]$. In either case the interval is contained in $I_B$, and
on it
\[
|h(s)|\ge \frac{\eta^2}{16L}.
\]
Thus in either case there is a subinterval of $I_B$ of length $\ell$ on which
\[
|h(s)|\ge
\eta_*:=\min\left\{\frac{\eta}{2},\frac{\eta^2}{16L}\right\}.
\]
Because $\tau\le \ell$, this subinterval contains at least
$\ell/(2\tau)$ grid points: if $\tau\le\ell/2$, then the count is at least
$\lfloor \ell/\tau\rfloor\ge \ell/\tau-1\ge \ell/(2\tau)$; if
$\ell/2<\tau\le\ell$, the subinterval contains at least one grid point and
$1\ge \ell/(2\tau)$. The full block contains at most
$2\eta/(L\tau)+1=16\ell/\tau+1\le 17\ell/\tau$ grid points. Hence at least a
$1/64$ fraction of the grid points in the block have $|h(u_j)|\ge\eta_*$.
If the block contains $N_B$ grid points, then
\[
\sum_{u_j\in I_B} |h(u_j)|^2
\ge \frac{N_B}{64}\eta_*^2
\ge
N_B\left(
\frac{\eta}{16}\min\left\{1,\frac{\eta}{8L}\right\}
\right)^2.
\]

It remains only to cover the whole grid by such blocks. Greedily split the
grid into consecutive blocks with time span at least $\eta/(2L)$ and stop each
block as soon as this is achieved. Each completed block has span at most
$\eta/(2L)+\tau\le 5\eta/(8L)$. If the final leftover block has shorter span,
merge it with the preceding block; the merged block has span at most
$2\eta/L$. Applying the preceding estimate block by block and summing the
squares gives the claimed bound on the full sliding window.
\end{proof}

\begin{claim}
\label{claim: Claim C.2}
Suppose $|f'|,|f''|\le L$ and
\[
m\tau\ge \frac{\epsilon}{4L},\qquad \tau\le \frac{\epsilon}{16L}.
\]

\begin{enumerate}
\item[(a)]  If $f$ is $(d,\epsilon)$-non-periodic on $J_{m,\tau}$, then for
any $t_{1},t_{2}\in\mathbb{T}$ with $|t_{1}-t_{2}|\ge d$, 
\[
\left\Vert SW_{m,\tau}f(t_{1})-SW_{m,\tau}f(t_{2})\right\Vert _{2}\ge\sqrt{m+1}\,\frac{\epsilon}{16}\min\left\{ 1,\frac{\epsilon}{16L}\right\} .
\]
\item[(b)]  If $f$ is $(\Xi,d,\epsilon)$-periodic, then for any $t_{1},t_{2}\in\mathbb{R}$
with $\min_{n\in\mathbb{Z}}\vert t_{1}-t_{2}+n\Xi\vert\ge d$, 
\[
\left\Vert SW_{m,\tau}f(t_{1})-SW_{m,\tau}f(t_{2})\right\Vert _{2}\ge\sqrt{m+1}\,\frac{\epsilon}{16}\min\left\{ 1,\frac{\epsilon}{16L}\right\} .
\]
\end{enumerate}

\end{claim}

\begin{proof}
Set $g(u)=f(t_1+u)-f(t_2+u)$. In the non-periodic case,
$|t_1-t_2|\ge d$ implies that, for every $u\in[0,m\tau]$,
\[
|g(u)|\ge\epsilon\quad\text{or}\quad |g'(u)|\ge\epsilon.
\]
 In the periodic case, $\min_{n \in \mathbb{Z}}\vert (t_1+u)-(t_2+u)+n\Xi\vert=\min_{n\in \mathbb{Z}}\vert t_1- t_2+n\Xi\vert$, so again, for every $u\in[0,m\tau]$, 
\[
|g(u)|\ge\epsilon\qquad\text{or}\qquad|g'(u)|\ge\epsilon.
\]
Also $|g'|\le 2L$, and $g'$ is $2L$-Lipschitz. Applying
Claim~\ref{claim:time_delay_norm_lower} to $g$ gives
\[
\|SW_{m,\tau}g(0)\|_2
\ge
\sqrt{m+1}\,
\frac{\epsilon}{16}
\min\left\{1,\frac{\epsilon}{16L}\right\}.
\]
Since $SW_{m,\tau}g(0)=SW_{m,\tau}f(t_1)-SW_{m,\tau}f(t_2)$, the claim
follows.
\end{proof}

\begin{claim}
\label{claim: Claim C.3}
Suppose Assumption~\ref{ass:fp_fpp_nonzero} holds. If
\[
m\tau\ge \frac{\delta}{2L},\qquad \tau\le \frac{\delta}{8L},
\]
then
\[
\|SWf_{m,\tau}'(t)\|_2
\ge
\sqrt{m+1}\,\nu_{\delta,L}
\qquad\text{for all relevant }t.
\]
\end{claim}

\begin{proof}
Since $(SW_{m,\tau}f(t))'=SW_{m,\tau}(f')(t)$, apply
Claim~\ref{claim:time_delay_norm_lower} to $h=f'$ with $\eta=\delta$.
The assumptions needed for that claim are exactly
$|f'|\ge\delta$ or $|f''|\ge\delta$, $|f''|\le L$, and $f''$ is
$L$-Lipschitz.
\end{proof}

\begin{claim}
\label{claim: Claim C.4}
Let $\gamma$ be the arc-length parametrization of $SW_{m,\tau}f$.
Under the assumptions of Claim~\ref{claim: Claim C.3},
\[
\|\gamma''\|_2
\le
\frac{2L}{\sqrt{m+1}\nu_{\delta, L}^2}.
\]
\end{claim}

\begin{proof}
For any regular parametrized curve $\Gamma(t)$, its arc-length
parametrization satisfies
\[
\|\gamma''\|_2
\le
\frac{2\|\Gamma''(t)\|_2}{\|\Gamma'(t)\|_2^2}.
\]
Here
\[
\|\Gamma_{m,\tau}''(t)\|_2
=
\left(\sum_{j=0}^{m} f''(t+j\tau)^2\right)^{1/2}
\le L\sqrt{m+1}.
\]
Together with Claim~\ref{claim: Claim C.3}, this gives
\[
\|\gamma''\|_2
\le
\frac{2L\sqrt{m+1}}{(m+1)\nu_{\delta,L}^2} = \frac{2L}{\sqrt{m+1}\nu_{\delta, L}^2}.
\]
\end{proof}

For a curve $\gamma:I\subset\mathbb{R}\to\mathbb{R}^{m+1}$, let $d_{\gamma}:I\times I\to\mathbb{R}$
be geodesic distance on $\gamma$. That is, for $s_{1},s_{2}\in I$,
\[
d_{\gamma}(s_{1},s_{2})=\inf_{s_{1}^{\prime}=s_{1},s_{2}^{\prime},\cdots,s_{n-1}^{\prime},s_{n}^{\prime}=s_{2}}\sum_{i=1}^{n-1}\left\Vert \gamma(s_{i+1}^{\prime})-\gamma(s_{i}^{\prime})\right\Vert _{2}.
\]
If we assume that $\gamma$ is $C^{1}$ and unit speed, i.e., $\left\Vert \gamma'(s)\right\Vert _{2}=1$
for all $s\in I$, then this is equivalent to that 
\[
d_{\gamma}(s_{1},s_{2})=\min_{\gamma(s_{1}^{\prime})=\gamma(s_{1}),\gamma(s_{2}^{\prime})=\gamma(s_{2})}\left|s_{1}^{\prime}-s_{2}^{\prime}\right|.
\]

\begin{lemma}[Short-arc chord bound]
\label{lem:short_arc_chord}
Let $\gamma$ be a unit-speed $C^2$ curve with $\|\gamma''\|\le 1/R$. 
If $0<d_{\gamma}(s_{1},s_{2})<\pi R$, then 
\[
\left\Vert \gamma(s_{1})-\gamma(s_{2})\right\Vert _{2}\ge2R\sin\left(\frac{d_{\gamma}(s_{1},s_{2})}{2R}\right)>0.
\]
\end{lemma}

\begin{proof}
This is the standard higher-dimensional Schur comparison for curves with
curvature bounded by $1/R$. Among unit-speed curves with curvature at most
$1/R$ and fixed intrinsic length less than $\pi R$, the circular arc of
radius $R$ has the smallest endpoint chord.
\end{proof}

\begin{lemma}[Curve reach criterion]
\label{lem:curve_reach_criterion}
Let $\gamma$ be a unit-speed $C^2$ embedded compact curve, either an interval
or a circle, in Euclidean space. Let $d_\gamma$ denote intrinsic distance, with
circular intrinsic distance in the closed case. Suppose that
\[
\|\gamma''\|\le \frac{1}{R}
\]
and that every pair with $d_\gamma(s_1,s_2)\ge \pi R$ satisfies
\[
\|\gamma(s_1)-\gamma(s_2)\|_2\ge 2R'.
\]
Then
\[
\operatorname{reach}(\gamma)\ge \min\{R,R'\}.
\]
\end{lemma}

\begin{proof}
We use the standard reach characterization for $C^2$ curves: the reach is the
minimum of the reciprocal maximal curvature and one half of the minimal
Euclidean distance over doubly critical pairs; see, for example,
\cite{aamari2019estimatingreachmanifold}. For compact arcs with boundary,
the same criterion is applied with one-sided tangent cones at the endpoints;
equivalently, one may extend the $C^2$ arc slightly past its endpoints and
then restrict the resulting tubular neighborhood to the original compact arc.
Endpoints therefore do not introduce a smaller reach scale than the curvature
and global normal-chord scales below. The curvature condition gives the first
lower bound $R$. A doubly critical pair with intrinsic separation strictly
smaller than $\pi R$ is excluded by the comparison theorem for curves with
curvature bounded by $1/R$; the model case is a circle of radius $R$, where
the first normal chord occurs at intrinsic separation $\pi R$. Hence any
doubly critical pair has intrinsic separation at least $\pi R$, and therefore
has Euclidean distance at least $2R'$. The reach is thus at least
$\min\{R,R'\}$.
\end{proof}

\begin{lemma}
\label{lem:nonperiodic_reach}
Under the assumptions of Theorem~\ref{thm:nonperiodic_contractible},
$\Gamma_{m,\tau}(\mathbb{T})$ is an embedded interval and
\[
\operatorname{reach}(\Gamma_{m,\tau}(\mathbb{T}))\ge a_{m,\tau}.
\]
\end{lemma}

\begin{proof}
Let $\gamma:[0,S_{\max}]\to\mathbb{R}^{m+1}$ be the arc-length
parametrization of $SW_{m,\tau}f(\mathbb{T})$.

Claim~\ref{claim: Claim C.4} gives the curvature condition with
\[
R_0=\sqrt{m+1}\,\frac{\nu_{\delta,L}^2}{2L}.
\]
Also $\|SW_{m,\tau}f'(t)\|_2\le L\sqrt{m+1}$. We first show that the
curve is embedded. If $0<|s_1-s_2|<\pi R_0$, then
Lemma~\ref{lem:short_arc_chord} gives
\[
\|\gamma(s_1)-\gamma(s_2)\|_2>0.
\]
If $|s_1-s_2|\ge \pi R_0$, then their original time parameters satisfy
\[
|t(s_1)-t(s_2)|
\ge
\frac{|s_1-s_2|}{L\sqrt{m+1}}
\ge
\frac{\pi \nu_{\delta, L}^2}{2L^2}
=d_{\delta,L}>d.
\]
By Claim~\ref{claim: Claim C.2},
\[
\|\gamma(s_1)-\gamma(s_2)\|_2
\ge
2\sqrt{m+1}\,\frac{\epsilon}{32} \min\bigg\{1,\frac{\epsilon}{16L}\bigg\}.
\]
Thus distinct arc-length parameters have distinct images.

Lemma~\ref{lem:curve_reach_criterion} gives
\[
\operatorname{reach}(\Gamma_{m,\tau}(\mathbb{T}))
\ge
\sqrt{m+1}\min\bigg\{\frac{\epsilon}{32}\min\big\{1,\frac{\epsilon}{16L}\big\},\frac{\nu_{\delta,L}^2}{2L}\bigg\}
=a_{m,\tau}.
\]
Hence the image is an embedded copy of the interval $\mathbb{T}$.
\end{proof}

\begin{proof}[Proof of Theorem~\ref{thm:nonperiodic_contractible}]
By Lemma~\ref{lem:nonperiodic_reach}, the support is an embedded interval with
reach at least $a_{m,\tau}$. For $0<r<a_{m,\tau}$, the nearest-point projection
onto the support is uniquely defined on the offset. The straight-line homotopy
from each point to its nearest point is a deformation retraction. Hence the
offset deformation retracts to an interval and is contractible.
\end{proof}

\begin{lemma}
\label{lem:periodic_reach}
Under the assumptions of Theorem~\ref{thm:periodic_circle},
$\Gamma_{m,\tau}([0,\Xi])$ is an embedded circle and
\[
\operatorname{reach}(\Gamma_{m,\tau}([0,\Xi]))\ge a_{m,\tau}.
\]
\end{lemma}

\begin{proof}
Because $f$ is $\Xi$-periodic, $\Gamma_{m,\tau}(0)=\Gamma_{m,\tau}(\Xi)$.
Let $\gamma:[0,S_{\max}]\to\mathbb{R}^{m+1}$ be the arc-length
parametrization over one period. The curvature bound is the same as in the
non-periodic case, with $R_0=\sqrt{m+1}\,\frac{\nu_{\delta,L}^2}{2L}$.

We first prove injectivity on $\mathbb{R}/\Xi\mathbb{Z}$. If the circular
arc-length distance between two parameters is in $(0,\pi R_0)$, then
Lemma~\ref{lem:short_arc_chord} gives positive Euclidean distance. If the
circular arc-length distance is at least $\pi R_0$, then the corresponding
circular time distance satisfies
\[
d_\Xi(t(s_1),t(s_2))
\ge
\frac{\pi R_0}{L\sqrt{m+1}}
=d_{\delta,L}>d.
\]
Claim~\ref{claim: Claim C.2} then gives Euclidean separation at least
$2\sqrt{m+1}\,\cdot (\epsilon/32) \cdot\min\{1,(\epsilon/16L)\}$. Hence the map from
$\mathbb{R}/\Xi\mathbb{Z}$ to its image is injective, so the image is an
embedded circle. Lemma~\ref{lem:curve_reach_criterion} gives reach at least
$a_{m,\tau}$.
\end{proof}

\begin{proof}[Proof of Theorem~\ref{thm:periodic_circle}]
Since $T_{\max}-T_{\min}\ge \Xi$ and $f$ is $\Xi$-periodic,
$\Gamma_{m,\tau}$ is $\Xi$-periodic in its time parameter. Hence the image of
any observation interval of length at least $\Xi$ equals the image over one
period, so
$\Gamma_{m,\tau}(\mathbb{T})=\Gamma_{m,\tau}([0,\Xi])$. By
Lemma~\ref{lem:periodic_reach}, this support is an embedded circle with reach
at least $a_{m,\tau}$. For $0<r<a_{m,\tau}$, the nearest-point projection
gives a deformation retraction of the offset onto the circle.
\end{proof}

\section{Proof of Theorem~\ref{thm:nonStandconfidence}}

To prove Theorem \ref{thm:nonStandconfidence} we verify the assumptions required by Proposition \ref{prop:curve_boundary_subsampling}. This proposition covers both the periodic circle case and the non-periodic interval case.

\begin{proof}[Proof of Proposition~\ref{prop:curve_boundary_subsampling}]
The proof is the Method I subsampling argument of \cite{Fasy_2014}, with the
``without boundary'' manifold assumption replaced by the stated local mass
condition and the one-dimensional packing estimate for compact curves. The
stability reduction in Section~4 of \cite{Fasy_2014} only uses the inequality
between bottleneck distance and Hausdorff distance for distance functions, and
therefore does not depend on whether the support has boundary. The boundary
issue enters only through the small-ball estimates used in the subsampling
step; for a compact embedded interval these estimates are one-sided at the two
endpoints and hence differ from the circle or interior case only by constants.
Indeed, because $M$ is a compact $C^2$ embedded curve with positive reach,
there are constants $0<C_1<C_2<\infty$ such that, for all sufficiently small
$r>0$, every maximal $r$-packing and every minimal $r$-covering of $M$ have
cardinality between $C_1r^{-1}$ and $C_2r^{-1}$. This follows from finite
curve length and the uniform chord--arc control implied by the positive reach.

The local mass condition gives
\[
\rho:=\liminf_{t\downarrow 0}\inf_{x\in M}\frac{P(B(x,t/2))}{t}>0.
\]
Consequently, for $t_n=C\log n/n$ with $C$ large enough,
\[
\mathbb{P}\{d_H(S_n,M)>t_n\}=o(1).
\]
On the event $\{d_H(S_n,M)<t_n\}$, the same packing argument used in the
proof of Theorem~3 of \cite{Fasy_2014} gives
\[
d_H(S_n,M)\le d_H(S_{b,n}^{(j)},S_n)
\]
simultaneously for every size-$b$ subsample, since
$b=o(n/\log n)$ while the number of $t_n$-packing cells is of order
$n/\log n$. The remainder of the subsampling proof compares the empirical
subsample tail with the corresponding i.i.d. $b$-sample tail, giving the same
$O((b/n)^{1/4})$ error term. Finally, bottleneck stability for distance
functions gives
\[
d_B(\mathcal{P}(S_n),\mathcal{P}(M))\le d_H(S_n,M),
\]
which yields the claimed confidence bound.
\end{proof}

\begin{lemma}
\label{lem:Assumption_Checking}
Under the assumptions of Theorem \ref{thm:nonStandconfidence}, $M_{m,\tau}$ satisfies:

(1) $\mathrm{reach}(M_{m,\tau})>0$

(2) For
$$\rho(x,t) = \frac{P(B(x, t/2))}{t},\quad x \in M_{m,\tau}$$ 
there exists $t_0 > 0$ such that $\rho(x,t)$ is differentiable for $t \in (0, t_0)$ and right differentiable at zero. Moreover, $\rho(x,t)$ is bounded above and away from zero, and there exists $C<\infty$ such that
\[\sup_x \sup_{0\le t\le t_0} \bigg\vert \frac{\partial \rho(x,t)}{\partial t}\bigg\vert \le C < \infty\]
\end{lemma}

\begin{proof}
    (1) By our results in Section \ref{sec:top_period}, for the functions satisfying Assumption \ref{ass:periodic_either}, \ref{ass:fp_fpp_nonzero}, and \ref{ass:d inequality}, and $m,\tau$ satisfying inequalities in Theorem \ref{thm:nonStandconfidence}, $M_{m,\tau}$ is a manifold with positive reach.

    \bigskip
    (2) Let $\Gamma(t) = SW_{m,\tau}f(t)$. Claim \ref{claim: Claim C.3} and Assumption \ref{ass:fp_fpp_nonzero} give,
    \[\nu_{\delta, L}\sqrt{m+1} \le\Vert\Gamma'(t)\Vert_2 \le L\sqrt{m+1}, \qquad \Vert \Gamma''(t)\Vert_2 \le L\sqrt{m+1}.\]

    Choose $t_0 < \text{reach}(M_{m,\tau})$ small enough that every
    $B(\Gamma(s),t/2)\cap M_{m,\tau}$ is a single connected arc for
    $0<t<t_0$, with a one-sided arc allowed when $\Gamma(s)$ is an endpoint
    in the non-periodic interval case. For $x=\Gamma(s)$, write
    \[
    A_{s,t}:=\{u:\Gamma(u)\in B(x,t/2)\}.
    \]
    In the periodic case the parameter $u$ is taken on
    $\mathbb{T}/\Xi\mathbb{Z}$; in the non-periodic case it is taken on the
    interval $\mathbb{T}=[T_{\min},T_{\max}]$.

    Then the endpoints of these arcs are $C^1$ functions of $t$ by the inverse function theorem applied to $u \mapsto \Vert \Gamma(s+u) - \Gamma(s) \Vert_2$ because the speed is bounded below and the curve has positive reach.

    The lower and upper speed bounds imply that the length of $A_{s,t}$ is comparable to $t$, uniformly in $s$ and $t < t_0$. Since the sampling density is assumed to be bounded above and below by Assumption \ref{ass:Sampling_Density}, the $P$-mass of the set $A_{s,t}$ is comparable to $t$. Thus $\rho(x,t)$ is bounded above and away from zero.

    It remains to control the derivative of $\rho$. Each local mass has the form
    $$P(B(\Gamma(s), t/2)) = \int_{a_{-}(s,t)}^{a_+ (s,t)} q(u)du$$
    where we collapse the sampling density to $\mathbb{T}/\Xi\mathbb{Z}$ in the periodic case. At an interior point of the interval, and everywhere in the periodic quotient, the endpoint functions $a_{\pm} (s,t)$ are $C^1$ in $t$ for $0 \le t\le t_0$, and the $C^2$ curve bound gives the uniform expansions
    $$a_{\pm}(s,t) = s\pm \frac{t}{2\Vert \Gamma'(s)\Vert} + O(t^2),\qquad \partial_t a_{\pm}(s,t) = \pm \frac{1}{2\Vert \Gamma'(s)\Vert} + O(t).$$
    Since $q$ is Lipschitz by Assumption \ref{ass:Sampling_Density},
    $$P(B(\Gamma(s),t/2)) = \frac{q(s)}{\Vert \Gamma'(s)\Vert}t + O(t^2),\qquad \partial_t P(B(\Gamma(s),t/2)) = \frac{q(s)}{\Vert \Gamma'(s)\Vert} + O(t).$$
    At the left endpoint $s=T_{\min}$, the lower endpoint is fixed and
    \[
    a_+(s,t)=s+\frac{t}{2\|\Gamma'(s)\|}+O(t^2),
    \qquad
    \partial_t a_+(s,t)=\frac{1}{2\|\Gamma'(s)\|}+O(t),
    \]
    so
    \[
    P(B(\Gamma(s),t/2))
    =
    \frac{q(s)}{2\|\Gamma'(s)\|}t+O(t^2).
    \]
    The right endpoint is identical with the one-sided interval reversed.
    Thus the same boundedness and differentiability conclusions hold at the
    boundary, with the leading local mass constant divided by two.
    Therefore,
    $$\frac{\partial}{\partial t} \rho(x,t) = \frac{\partial_t P(B(\Gamma(s),t/2))}{t} - \frac{P(B(\Gamma(s),t/2))}{t^2}$$
    is uniformly bounded for $0 < t < t_0$ and has a one-sided limit at $t=0$.
\end{proof}

\medskip

\begin{proof}[Proof of Lemma \ref{lem:calpha_converge_zero}] 
Let $X_n = X_{m,\tau}$ and $X_b^{(I)}$ denote a uniformly chosen $b$-subsample of $X_n$, where the index set $I$ is chosen independently of the sample after conditioning on $X_n$. Let
$$P(r) = P_I (d_H(X_n, X_b^{(I)}) > r \mid X_n).$$
Taking expectations over the original sample gives
$$\mathbb{E}(P(r)) = P(d_H(X_n, X_b^{(I)})> r).$$

Although $X_b^I$ is not independent of $X_n$, its unconditional joint distribution is the same as that of $b$ i.i.d. samples from the pushforward distribution on $M_{m,\tau}$.

Fix $r >0$. Since $M_{m,\tau}$ is compact and sampling density is bounded below, there exist finitely many balls $B(y_1, r/4),\cdots, B(y_K, r/4)$ covering $M_{m,\tau}$ and $p_r > 0$ such that $P(B(y_k, r/4))\ge p_r$ for every $k$. For $N$ i.i.d. samples $Z_N$ from $P$,
$$P(d_H(Z_N, M_{m,\tau} )>r/2) \le K(1-p_r)^N \rightarrow 0.$$
Therefore by triangle inequality,
$$P(d_H(X_n, X_b^{(I)}) > r) \rightarrow 0$$
as $n,b\rightarrow \infty$. 

Thus $\mathbb{E} (P(r))\rightarrow 0$ for every $r>0$. By Markov's inequality, $P(r) \xrightarrow{P} 0$, and therefore
$$P(c_{\alpha} > \eta) \le P(P(\eta/2) > \alpha) \rightarrow 0.$$
Hence $c_{\alpha}\xrightarrow{P}0$.
\end{proof}

\begin{proof}[Proof of Theorem \ref{thm:nonStandconfidence}]
    The sampled sliding-window points $X_{m,\tau}$ are sampled from the support $M_{m,\tau}$. By Lemma \ref{lem:Assumption_Checking}, the support satisfies the regularity assumptions needed for Proposition \ref{prop:curve_boundary_subsampling}. Therefore, by the definition of $c_{\alpha}$,
    \[
        P\bigg(d_B(\mathcal{P}(X_{m,\tau}), \mathcal{P}(M_{m,\tau})) > c_{\alpha} \bigg) \le \alpha + O\bigg(\frac{b}{n}\bigg)^{1/4}.
    \]
\end{proof}

\section{Proof of Section~\ref{sec:linearInterpolation}}

To prove Theorem \ref{thm:confidenceBoundInterpolation} first check the error term based on the assumption \ref{ass:errorInterpolation}.

\begin{lemma}\label{lem:interpolationBound}
    Under the assumption \ref{ass:errorInterpolation}, for the linear interpolation $\hat{f_N}$,
    \[\left\|\hat{f_N}-f\right\|_\infty\le \eta_N\]
    Consequently,
    \[
    d_H(\hat{X}_{m,\tau},X_{m,\tau})\le \zeta_{m,N},\qquad d_H(\hat{X}_{m,\tau}^{(j)},X_{m,\tau}^{(j)})\le \zeta_{m,N}
    \]
    for every subsample, and
    \[
    d_B(\mathcal{P}(\hat{X}_{m,\tau}),\mathcal{P}(X_{m,\tau}))\le \zeta_{m,N}
    \]
\end{lemma}

\begin{proof}
    The linear interpolation $\hat{f_N}$ of the sample $(s_i,Y_i)$ is
    \begin{align*}
        \hat{f_N}(t)&=\frac{s_{i+1}-t}{s_{i+1}-s_i}Y_i+\frac{t-s_{i}}{s_{i+1}-s_i}Y_{i+1}\\
        &=\frac{s_{i+1}-t}{s_{i+1}-s_i}f(s_i)+\frac{t-s_{i}}{s_{i+1}-s_i}f(s_{i+1})+\frac{(s_{i+1}-t)\xi_i+(t-s_i)\xi_{i+1}}{s_{i+1}-s_i}.
    \end{align*}
    For the error $|\hat{f_N}(t)-f(t)|$, we get an error term from each interpolation and noise.
    
    The noise error is bounded by assumption \ref{ass:errorInterpolation}
    \[
    \left|\frac{(s_{i+1}-t)\xi_i+(t-s_i)\xi_{i+1}}{s_{i+1}-s_i}\right|\le \max_j|\xi_j|\le \sigma_N
    \]
    For the interpolation error, we get using Taylor series.
    
    For some $s_i<r_i<t<r_{i+1}<s_{i+1}$,
    \[
    f(s_i)=f(t)+f'(t)(s_i-t)+\frac{f''(r_i)}{2}(s_i-t)^2
    \]
    and
    \[
    f(s_{i+1})=f(t)+f'(t)(s_{i+1}-t)+\frac{f''(r_{i+1})}{2}(s_{i+1}-t)^2
    \]
    \begin{align*}
        &\left|\frac{s_{i+1}-t}{s_{i+1}-s_i}f(s_i)+\frac{t-s_{i}}{s_{i+1}-s_i}f(s_{i+1})-f(t)\right|\\
        &=\left|\left(\frac{s_{i+1}-t}{s_{i+1}-s_i}+\frac{t-s_{i}}{s_{i+1}-s_i}-1\right)f(t)\right.\\
        &+\left(\frac{(s_{i+1}-t)(s_i-t)+(t-s_i)(s_{i+1}-t)}{s_{i+1}-s_i}\right)f'(t)\\
        &+\left.\frac{(s_{i+1}-t)f''(r_{i})(t-s_i)^2+(t-s_i)f''(r_{i+1})(s_{i+1}-t)^2}{2(s_{i+1}-s_i)}\right|\\
        &=\left| \frac{(s_{i+1}-t)(t-s_i)}{2}\frac{f''(r_i)(t-s_i)+f''(r_{i+1})(s_{i+1}-t)}{s_{i+1}-s_i}\right|\\
        &\le \left| \frac{(s_{i+1}-t)(t-s_i)}{2}\right|\max_a|f''(a)|
    \end{align*}
    Under the assumption \ref{ass:errorInterpolation}, we know that $\Delta_N=\max_{0\le i<N} (s_{i+1}-s_i)$ and $|f''(t)|\le L_{int}$ for all $t\in J_{m,\tau}$.
    Hence, 
    \[
    \|\hat{f_N}-f\|_{\infty}\le \sup_{t\in J_{m,\tau}}|\hat{f_N}(t)-f(t)|\le \frac{\Delta_N^2}{8}L_\text{int} + \sigma_N=\eta_N
    \]
    gives all $m+1$ components bounded by $\eta_N$, so
    \[
    \|SW_{m,\tau}\hat{f_N}(t)-SW_{m,\tau}f(t)\|_2\le \sqrt{m+1}\eta_N=\zeta_{m,N}.
    \]
    Since each point is perturbed by at most $\zeta_{m,N}$, the Hausdorff distance satisfies
    \[
    d_H(\hat{X}_{m,\tau},X_{m,\tau})\le \zeta_{m,N},\qquad d_H(\hat{X}_{m,\tau}^{(j)},X_{m,\tau}^{(j)})\le \zeta_{m,N}.
    \]
    Also, by stability theorem, the bottleneck distance is bounded by
    \[
    d_B(\mathcal{P}(\hat{X}_{m,\tau}),\mathcal{P}(X_{m,\tau}))\le d_H(\hat{X}_{m,\tau},X_{m,\tau})\le \zeta_{m,N}.
    \]
\end{proof}

\begin{lemma}\label{lem:cAlphaBound}
Under the assumption \ref{ass:errorInterpolation},

\[
c_\alpha \le \hat{c_\alpha} + 4 \zeta_{m,N},\qquad \hat{c_\alpha} \le c_\alpha + 4 \zeta_{m,N}.
\]
\end{lemma}

\begin{proof}
By the triangle inequality, for every subsample index $j$,
\begin{align*}
&\left| d_H(\hat{X}_{m,\tau}, \hat{X}^{(j)}_{m,\tau,b})
- d_H(X_{m,\tau}, X^{(j)}_{m,\tau,b}) \right| \\
&\le d_H(\hat{X}_{m,\tau}, X_{m,\tau})
+ d_H(\hat{X}^{(j)}_{m,\tau,b}, X^{(j)}_{m,\tau,b})
\le 2\zeta_{m,N}.
\end{align*}
For fixed $t > 0$, $d_H(\hat{X}_{m,\tau}, \hat{X}^{(j)}_{m,\tau,b}) > t + 2\zeta_{m,N}$ implies that $d_H(X_{m,\tau}, X^{(j)}_{m,\tau,b}) > t$. Hence,
\[
I(d_H(\hat{X}_{m,\tau}, \hat{X}^{(j)}_{m,\tau,b}) > t+ 2\zeta_{m,N})\le I(d_H(X_{m,\tau}, X^{(j)}_{m,\tau,b}) > t)
\]
Taking average over the subsample index $j$, we obtain a bound for shifted $\hat{L}$ as
\[
\hat{L}_{m,\tau}(t + 2\zeta_{m,N}) \le L_{m,\tau}(t), 
\]
Similarly, we can say $d_H(X_{m,\tau}, X^{(j)}_{m,\tau,b}) > t + 2\zeta_{m,N}$ implies $d_H(\hat{X}_{m,\tau}, \hat{X}^{(j)}_{m,\tau,b}) > t$, so
\[
L_{m,\tau}(t + 2\zeta_{m,N}) \le \hat{L}_{m,\tau}(t).
\]

Let $\delta = 2\zeta_{m,N}$ and let $\epsilon > 0$. By the generalized inverse convention, there exists $s \le \hat{L}_{m,\tau}^{-1}(\alpha) + \epsilon$ such that $\hat{L}_{m,\tau}(s) \le \alpha$. Hence $L_{m,\tau}(s + \delta) \le \alpha$, and therefore
\[
L_{m,\tau}^{-1}(\alpha) \le \hat{L}_{m,\tau}^{-1}(\alpha) + \epsilon + \delta.
\]
Sending $\epsilon \to 0$ gives
\[
L_{m,\tau}^{-1}(\alpha) \le \hat{L}_{m,\tau}^{-1}(\alpha) + \delta.
\]

The reverse inequality follows identically from the other shifted tail bound: for every $\epsilon > 0$, choose $s \le L_{m,\tau}^{-1}(\alpha) + \epsilon$ such that $L_{m,\tau}(s) \le \alpha$, then $\hat{L}_{m,\tau}(s + \delta) \le \alpha$. Thus
\[
\hat{L}_{m,\tau}^{-1}(\alpha) \le L_{m,\tau}^{-1}(\alpha) + \delta.
\]
It follows that $|\hat{L}_{m,\tau}^{-1}(\alpha) - L_{m,\tau}^{-1}(\alpha)| \le \delta=2\zeta_{m,N}$. Thus, $|\hat{c}_\alpha/2-c_\alpha/2|\le 2\zeta_{m,N}$ i.e. $|\hat{c}_\alpha-c_\alpha|\le 4\zeta_{m,N}$.

\end{proof}

\begin{proof}[Proof of Theorem \ref{thm:confidenceBoundInterpolation}]
Let
\[
E_\alpha := \{ d_B(\mathcal{P}(X_{m,\tau}), \mathcal{P}(M_{m,\tau})) \le c_\alpha \}.
\]
By Theorem \ref{thm:nonStandconfidence},
\[
P(E_\alpha^c) \le \alpha + O \left(\frac{b}{n}\right)^{1/4} .
\]

On $E_\alpha$, Lemmas \ref{lem:interpolationBound} and \ref{lem:cAlphaBound} give
\begin{align*}
    d_B(\mathcal{P}(\hat{X}_{m,\tau}), \mathcal{P}(M_{m,\tau}))&\le d_B(\mathcal{P}(\hat{X}_{m,\tau}), \mathcal{P}(X_{m,\tau})) + d_B(\mathcal{P}(X_{m,\tau}), \mathcal{P}(M_{m,\tau})) \\
    &\le \zeta_{m,N} + c_\alpha \\
    &\le \hat{c}_\alpha + 5\zeta_{m,N} = \hat{c}_\alpha^{\mathrm{err}}.
\end{align*}

Thus the failure probability is bounded by $P(E_\alpha^c)$.

\end{proof}

\begin{proof}[Proof of Corollary \ref{cor:limitInterpolationError}]
By Lemma \ref{lem:cAlphaBound},
\[
\hat{c}_\alpha^{\mathrm{err}} = \hat{c}_\alpha + 5\zeta_{m,N} \le c_\alpha + 9\zeta_{m,N}.
\]

Lemma \ref{lem:calpha_converge_zero} gives $c_\alpha \to 0$ in probability, and $\zeta_{m,N} \to 0$ by assumption. Hence $\hat{c}_\alpha^{\mathrm{err}} \to 0$ in probability.

\end{proof}

\section{Proofs for Section~\ref{sec:testing}}

\begin{proof}[Proof of Theorem~\ref{thm:testing_error_bound}]
Let
\[
E_\alpha
:=
\left\{
d_B(\mathcal{P}(Y),\mathcal{P}(M_{m,\tau}))
\le r_\alpha
\right\}.
\]
In the exact case, Theorem~\ref{thm:nonStandconfidence} gives
$P(E_\alpha^c)\le \alpha+O((b/n)^{1/4})$; in the bounded-interpolation case,
the same bound follows from Theorem~\ref{thm:confidenceBoundInterpolation}.

Under $H_0$, suppose $E_\alpha\cap A_\alpha^r$ holds and
$p=(b_p,d_p)\in\mathcal{P}(Y)\cap R(r_\alpha)$. Then
$b_p\le r_\alpha$ and $d_p-b_p>2r_\alpha$. Corollary~\ref{cor:non-periodic_homology}
implies that every off-diagonal support point
$q=(b_q,d_q)\in\mathcal{P}(M_{m,\tau})$ satisfies
$b_q\ge a_{m,\tau}>4r_\alpha$, so
$b_q-b_p>3r_\alpha$. Thus $p$ cannot be matched to any off-diagonal support
point within $r_\alpha$. It also cannot be matched to the diagonal within
$r_\alpha$, since $d_\infty(p,\Delta)=(d_p-b_p)/2>r_\alpha$. This contradicts
$E_\alpha$. Hence
\[
\{\psi(Y)=1,A_\alpha^r\}\subseteq E_\alpha^c
\]
under $H_0$.

Under $H_1$, Corollary~\ref{cor:periodic_homology} gives a support-level point
$q_0=(0,D)$ with $D\ge a_{m,\tau}$. On $A_\alpha^r$,
$D\ge a_{m,\tau}>4r_\alpha$, so $q_0$ is farther than $r_\alpha$ from the
diagonal. Hence, on $E_\alpha$, it must be matched to some
$p=(b_p,d_p)\in\mathcal{P}(Y)$ with
$|b_p|\le r_\alpha$ and $|d_p-D|\le r_\alpha$. Since births are nonnegative
for the distance filtration, $b_p\le r_\alpha$, and
\[
d_p-b_p\ge D-2r_\alpha\ge a_{m,\tau}-2r_\alpha>2r_\alpha.
\]
Thus $p\in R(r_\alpha)$ and $\psi(Y)=1$. Therefore
\[
\{\psi(Y)=0,A_\alpha^r\}\subseteq E_\alpha^c
\]
under $H_1$. The two conditional bounds follow from the bound on
$P(E_\alpha^c)$, and the unconditional consequences follow by decomposing
according to $A_\alpha^r$ and its complement.
\end{proof}

\begin{proof}[Proof of Corollary~\ref{cor:testing_asymptotic}]
For the exact rule, Lemma~\ref{lem:calpha_converge_zero} gives
$c_\alpha\xrightarrow{P}0$. For the bounded-interpolation rule,
Corollary~\ref{cor:limitInterpolationError} gives
$\hat c_\alpha^{\mathrm{err}}\xrightarrow{P}0$ whenever
$\zeta_{m,N}\to0$. Since $m,\tau$ are fixed, the topological results give
$a_{m,\tau}>0$, and therefore
$P((A_\alpha^r)^c)=P(a_{m,\tau}\le4r_\alpha)\to0$ in either case. Combining
this with Theorem~\ref{thm:testing_error_bound} gives the stated bounds.
\end{proof}

\begin{proof}[Proof of Corollary~\ref{cor:Finite_grid_testing}]
Under $H_0$, on the event $\bigcap_{l=1}^K A_{\alpha/K,l}^r$, the aggregate
test rejects only if at least one individual test rejects. The union bound and
Theorem~\ref{thm:testing_error_bound} applied at level $\alpha/K$ give
\begin{align*}
P\left(\Phi_{\mathcal{G}}=1,\bigcap_{l=1}^K A_{\alpha/K,l}^r\mid H_0\right)
&\le
\sum_{l=1}^K
P(\psi_l=1,A_{\alpha/K,l}^r\mid H_0)\\
&\le
\sum_{l=1}^K
\left\{\frac{\alpha}{K}
+O\left(\frac{b}{n}\right)^{1/4}\right\}\\
&=
\alpha+O\left(K\left(\frac{b}{n}\right)^{1/4}\right).
\end{align*}

Under $H_1$, if candidate $l_*$ satisfies the periodic assumptions and the
aggregate rule fails to reject, then the individual rule for $l_*$ also fails
to reject. Hence
\[
\{\Phi_{\mathcal{G}}=0,A_{\alpha/K,l_*}^r\}
\subseteq
\{\psi_{l_*}=0,A_{\alpha/K,l_*}^r\}.
\]
Theorem~\ref{thm:testing_error_bound} applied at level $\alpha/K$ gives
\[
P\left(\Phi_{\mathcal{G}}=0,A_{\alpha/K,l_*}^r\mid H_1\right)
\le
\frac{\alpha}{K}+O\left(\frac{b}{n}\right)^{1/4}.
\]
The final asymptotic statement follows when the corresponding separation
events have probability tending to one.
\end{proof}

\section{Additional explanation for Section \ref{sec:simul}}
\label{sec:err_exp}

The synthetic code supports several scalar error models. For the bounded-error
experiments used with the robust TDA rule, the observed scalar values are
\[
Y_i=f(s_i)+\xi_i,\qquad |\xi_i|\le \sigma_N.
\]
We implement this either as bounded uniform noise
\[
\xi_i\sim \mathrm{Unif}[-\sigma_N,\sigma_N]
\]
or as a truncated Gaussian supported on the same interval. These are the cases
where the deterministic scalar error bound $\sigma_N$ is available. The
interpolation theorem also requires a smoothness constant for the underlying
signal; in the simulations below we use a dense-grid numerical estimate of this
constant when forming the reported error-inflated radius.

\subsection{Synthetic design and result generation}
\label{sec:synthetic_design}

The synthetic experiments are performed on the interval $[0,2\pi]$. The two
periodic alternatives are
\[
    f_1(t)=\frac{3}{2-\cos(5t)}, \qquad
    f_2(t)=5\log(\sin(10t)+5)+\exp(\cos(5t)),
\]
both with fundamental period $2\pi/5$. The non-periodic null signals are
\[
\begin{aligned}
    h_1(t)=(t-\pi)/\pi&, \qquad
    h_2(t)=((t-\pi)/\pi)^2,\\
    h_3(t)=\{1+\exp[-5(t-\pi)/\pi]\}^{-1}-1/2&, \qquad
    h_4(t)=0.45(t-\pi)/\pi+0.35\sin(1.7t+0.2).
\end{aligned}
\]
These null examples contain trend-like or low-frequency scalar structure
without forming a closed delay-space loop. For the bounded-error experiments,
we observe scalar values on the extended interval $[0,2\pi+m\tau]$, construct a
piecewise-linear interpolant, and use observed timestamps in $[0,2\pi]$ as
delay-vector anchors. The delay is chosen by the known-period rule
\[
\tau=\frac{2\pi/5}{m+1}.
\]

The synthetic results in Table~\ref{tab:synthetic_main_summary} were generated
by the command-line runner \texttt{simulation/run\_synthetic.py}. The runner
constructs the noisy scalar observations, builds the linear interpolant,
forms the raw delay cloud, computes the subsampling radius, and records both
the raw decision and the bounded-error decision based on
$\hat c_\alpha^{\mathrm{err}}=\hat c_\alpha+5\zeta_{m,N}$. It writes the full
replicate-level output, the grouped summary, and the exact configuration to
\texttt{raw\_results.csv}, \texttt{summary.csv}, and \texttt{config.json},
respectively.

The code computes
\[
    \zeta_{m,N}
    =\sqrt{m+1}\left(\sigma_N+\widehat L_{\mathrm{int}}\delta_N^2/8\right),
\]
where $\delta_N$ is the maximum scalar observation spacing and
$\widehat L_{\mathrm{int}}$ is a dense-grid numerical estimate of a second
derivative bound on the observed interval. Thus the robust synthetic results
should be read as the numerical implementation of the error-inflated rule, not
as using a separately certified analytic upper bound for
$L_{\mathrm{int}}$.

For the main bounded-error run, the script was run with the six signals
$f_1,f_2,h_1,\ldots,h_4$, 100 seeds, $N=400$ scalar observations, $m=20$,
$\tau=(2\pi/5)/(m+1)$, $n=336$ delay-vector anchors, $b=120$, 500 Monte Carlo
subsamples, and level $\alpha=0.05$. The table averages over seven bounded
noise settings: the noiseless case, bounded uniform noise with
$\sigma\in\{0.02,0.05,0.10\}$, and truncated Gaussian noise with
$\sigma\in\{0.02,0.05,0.10\}$. The robust TDA rule is reported only for these
bounded-error cases, where the deterministic scalar error bound $\sigma_N$ is
available. For unbounded noise models such as Gaussian or Laplace noise, the raw
TDA rule can still be computed, but the bounded-error theorem does not apply
without an additional high-probability error event.

\subsection{Near-periodic stress cases}

We also ran an appendix stress test on two nonstationary oscillatory signals,
\[
    g_1(t)=\sin(5t+0.18t^2), \qquad
    g_2(t)=\{1+0.18\cos(0.5t)+0.08t\}\sin(5t).
\]
Here, phase drift means that the oscillation phase is not linear in time, so
the local frequency changes slowly; amplitude modulation means that a sinusoid
is multiplied by a slowly varying, non-periodic envelope. 
These examples are not part of the main clean-null table because they may
produce substantial loop-like sliding-window geometry even though they are not
strictly periodic with period $2\pi/5$. Thus, TDA rejection here should be read
as detection of recurrent geometric structure, not as a type-I error under the
topological null used in Section~\ref{sec:testing}.

\begin{table}[!t]
\centering
\caption{Appendix stress-test rejection rates for near-periodic nonstationary
signals. Entries are averaged over 100 seeds and the same seven bounded-error
noise settings used in Table~\ref{tab:synthetic_main_summary}.}
\label{tab:synthetic_stress_summary}
\small
\begin{tabular}{llccccc}
\toprule
Signal & description & TDA raw & TDA robust & GLS & Fisher $g$ & Perm. \\
\midrule
$g_1$ & phase drift & 0.791 & 0.137 & 1.000 & 1.000 & 1.000 \\
$g_2$ & amplitude modulation & 1.000 & 0.714 & 1.000 & 1.000 & 1.000 \\
\bottomrule
\end{tabular}
\end{table}

Table~\ref{tab:synthetic_stress_summary} shows that the scalar spectral tests
reject all stress cases. The raw TDA rule also rejects frequently, especially
for the amplitude-modulated signal. The robust rule is more conservative
because the deterministic interpolation-error correction inflates the
rejection radius, but it can still reject when the delay cloud has a persistent
loop beyond that correction.

\subsection{Real-data protocol and diagnostic plots}
\label{sec:real_protocol}

The real-data analysis uses the respiration channel in both PhysioNet datasets.
It is local in time because respiratory periodicity may drift over long
intervals. For Fantasia, we use the first 90 seconds, downsample to 25 Hz, and
test seven overlapping 20-second windows with starting times
$0,10,\ldots,60$ seconds. For BIDMC, we use the first 180 seconds, apply a
$0.05$--$0.8$ Hz bandpass filter, downsample to 25 Hz, and test six 30-second
windows with starting times $0,30,\ldots,150$ seconds. Each processed window is
linearly detrended and standardized.

For each window, candidate periods are estimated from autocorrelation and Welch
spectral peaks. We take the strongest period hint and test up to three distinct
local candidates obtained by multiplying it by $0.9$, $1$, and $1.1$, clipped
to the allowed period range. The TDA embedding uses $m=8$ and
$\tau=\mathrm{period}/(m+1)$. Each delay cloud is capped at 300 anchor points,
and the subsampling confidence radius uses $b=0.99n$ with 300 Monte Carlo
subsamples. This large finite-sample subsample fraction improves sensitivity in
local windows but should be read as a practical calibration choice rather than
the asymptotic regime of Theorem~\ref{thm:nonStandconfidence}. We use global
level $\alpha=0.05$ and apply the finite-search Bonferroni correction from
Corollary~\ref{cor:Finite_grid_testing} within each record. The TDA correction
is over the tested window-period pairs, whereas the scalar baselines are run
once per window and are corrected over windows. A record is declared periodic if
at least one Bonferroni-corrected local test rejects. In the realized search
sets, this gives 825 TDA window-period tests for Fantasia and 954 for BIDMC; the
Fantasia count is not exactly $3\times 280$ because clipping and duplicate
candidates can reduce the number of distinct local periods. Since the period
candidates are estimated from the observed window, this should be read as a
Bonferroni correction over the realized finite search set; the formal statement
of Corollary~\ref{cor:Finite_grid_testing} applies to fixed finite grids.

For real data, we report only the raw TDA rule. The robust bounded-error rule
requires a deterministic or high-probability upper bound on $\zeta_{m,N}$, which
depends on both the scalar measurement error and the interpolation smoothness
constant. The real recordings alone do not provide either a certified scalar
measurement-error bound or a certified interpolation smoothness bound.

The real-data plotting script \texttt{simulation/plot\_real\_results.py}
also produces aggregate diagnostic plots from the record-level and window-level
summaries. Figure~\ref{fig:real_tda_diagnostics} reports the two diagnostics
that are most useful beyond Table~\ref{tab:real_summary}: the distribution of
the best TDA persistence margin over local windows and the number of rejecting
windows per record.
The plots show that the real-data rejections are local rather than uniform across all windows,
and that BIDMC has larger and more frequent positive TDA margins than Fantasia.

\begin{figure}[!t]
\centering
\includegraphics[width=\linewidth]{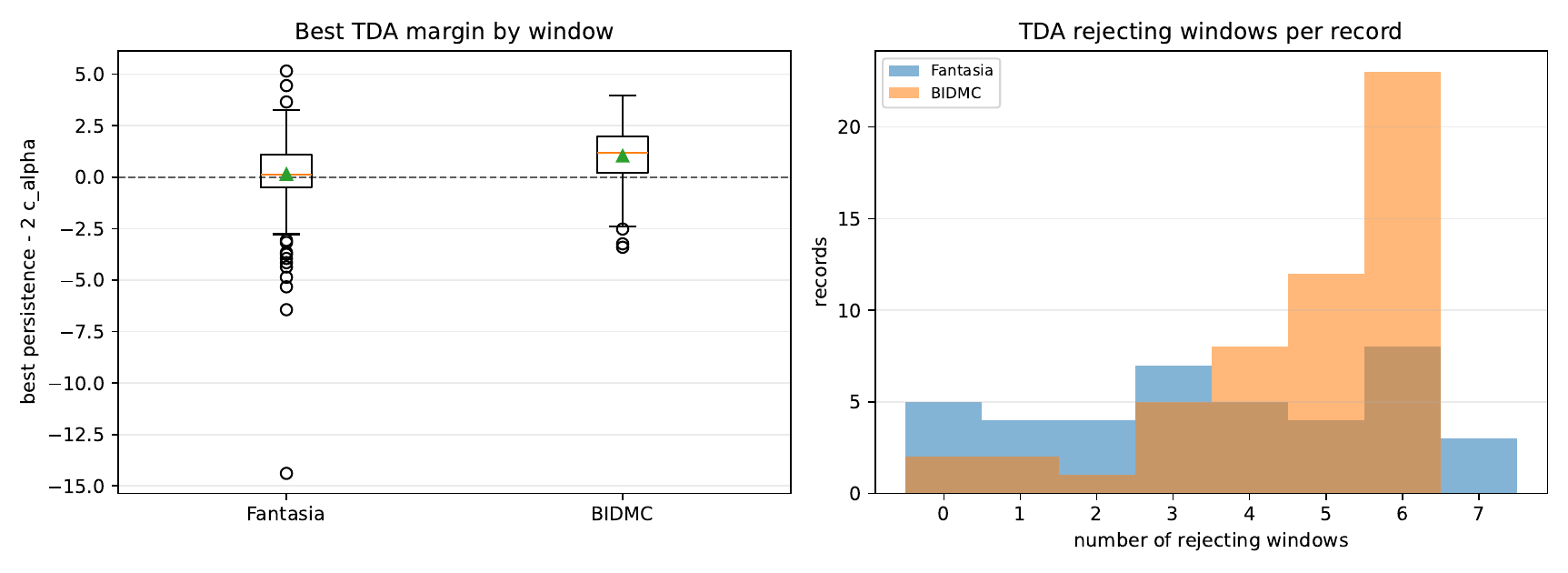}
\caption{Real-data TDA diagnostics. Left: for each local window, the plotted
quantity is the largest value of $\operatorname{pers}(\gamma)-2\hat c_\alpha$
over the tested candidate periods. Positive values show that the persistence
part of the TDA rule is satisfied for some candidate period; the full rejection
rule also requires the corresponding birth time to be at most $\hat c_\alpha$.
Right: distribution, over records, of the number of local windows with at least
one rejecting TDA test.}
\label{fig:real_tda_diagnostics}
\end{figure}


\end{document}